\documentclass[12pt]{iopart}
\usepackage{amsthm,amssymb}
\usepackage{graphicx,color}
\usepackage{subfigure}
\usepackage{multirow}
\usepackage{tikz}

\def\a{{\mathbf a}}

\def\x{{\mathbf x}}

\def\u{{\mathbf u}}
\def\v{{\mathbf v}}
\def\w{{\mathbf w}}

\def\n{{\mathbf n}}

\def\f{\frac}
\def\q{\quad}
\def\na{\nabla}
\def\Om{\Omega}
\def\om{\omega}
\def\p{\partial}

\def\eref#1{{(\ref{#1})}}
\newtheorem{theorem}{Theorem}[section]
\newtheorem{lemma}[theorem]{Lemma}

\newtheorem{proposition}[theorem]{Proposition}

\begin{document}
\title[]{Viscoelastic modulus reconstruction using time harmonic vibrations}
\author{Habib Ammari$^1$,  Jin Keun Seo$^2$ and Liangdong Zhou$^2$}
\address{$^1$ Department of Mathematics and Applications,
Ecole Normale Sup\'erieure, 45 Rue d'Ulm, 75005 Paris, France .}

\address{$^2$ Department of Computational Science and Engineering, Yonsei University, 120-749 Korea.}
\ead{habib.ammari@ens.fr, seoj@yonsei.ac.kr, zhould1990@hotmail.com}

\begin{abstract}
This paper presents a new iterative reconstruction method to provide high-resolution images of shear modulus and viscosity via the internal measurement of displacement fields in tissues.  To solve the inverse problem, we compute the Fr\'echet derivatives of the least-squares discrepancy functional  with respect to the shear modulus and shear viscosity. The proposed iterative reconstruction method using this Fr\'echet derivative does not require any differentiation of the displacement data for the full isotropic linearly viscoelastic model, whereas the standard reconstruction methods require at least double differentiation.  Because  the minimization problem is ill-posed and highly nonlinear, this adjoint-based optimization method needs a very well-matched initial guess.  We find a good initial guess. For a well-matched initial guess,  numerical experiments show that the proposed method considerably  improves the quality of the reconstructed viscoelastic images.
\end{abstract}

\section{Introduction}
Elastography \cite{Muthupillai1995} aims to provide a quantitative visualization of the mechanical properties of human tissues by using the relation between the wave propagation velocity and the mechanical properties of the tissues. During the last three decades, elastography led to significant improvements in the quantitative evaluation of tissue stiffness. The two major elastographic techniques are based on ultrasound \cite{Bercoff2003,Parker1990,Parker1992,Sandrin2002,Wu2006} and on magnetic resonance imaging  \cite{Manduca2001,Manduca2003, Miga2001,Muthupillai1995,
Oliphant2001,Park2006,Romano2000,Sack2007,Sinkus2005b,Sinkus2000}. GE Healthcare has recently  commercialized magnetic resonance elastography (MRE). Its main use is to assess mechanical changes in liver tissue.
The mechanical  properties of tissue include the shear modulus, shear viscosity, and compression modulus \cite{Landau1959}. Quantification of  the tissue shear modulus {\it in vivo} can provide evidence of the manifestation of tissue diseases.
For centuries, palpation has been widely used to identify tissue abnormalities and estimate  the mechanical properties of tissue. Therefore, it  is surprising  that  the concept of remote palpation, which is the remote imaging of tissue stiffness, was first developed only  in the late 1980s \cite{Hill2004,Krouskop1987,Sarvazyan2011}.

Although significant progress has been made in the development of  shear modulus imaging technology, problems remain image quality relating to the enhancement of images of local tissue shear viscosity and shear modulus \cite{Jiang2011,Kwon2009,Lee2010,McLaughlin2012,Sinkus2005a,Sinkus2005b,Song2012,Wall2011}. This paper focuses on the image reconstruction methods for tissue viscoelasticity imaging. To simplify the underlying inverse problem, the reconstruction of both the shear modulus and shear viscosity are considered under the assumption of isotropic elastic moduli.

This work considers the inverse problem of recovering the distribution of the shear modulus ($\mu$) and shear viscosity ($\eta_{\mu}$) from the internal measurement of the time-harmonic mechanical displacement field $\u$  produced by the application of  an external time harmonic excitation at frequency $\om/2\pi$ in the range $50\sim 200$Hz  through the surface of the subject. Modeling soft tissue as being linearly viscoelastic and nearly incompressible, the displacement $\u$ satisfies the elasticity equation
 \begin{equation}\label{Eq:viscoelasticity}
\na\cdot\left((\mu+i\omega\eta_{\mu})(\na\u+\na\u^t)\right)+\na((\lambda+i\omega\eta_{\lambda})\na\cdot\u)+\rho\omega^2\u=0, \end{equation}
where $\rho$ denotes the density of the medium, $\na\u^{t}$ is the transpose of the matrix $\na\u$, $\lambda$ is the compression modulus and $\eta_{\lambda}$  is the compression viscosity.

The most widely used reconstruction method is the algebraic inversion method \cite{Manduca2001}: For any non-zero constant vector $\a$,
\begin{equation}\label{Eq:directinversion}
 \mu+i\omega\eta_{\mu}=-\f{\rho\omega^2 (\a\cdot \u)}{\na\cdot\na (\a\cdot\u)},
 \end{equation}
 which requires the strong assumptions of $\na (\mu+i\omega\eta_{\mu})\approx 0$ (local homogeneity) and $(\lambda+i\omega\eta_{\lambda})\na\cdot\u\approx0$ (negligible pressure).

  The algebraic formula (\ref{Eq:directinversion})  ignores  reflection effects of the propagating wave due to  abrupt changes of $\mu+i\om\eta_{\mu}$, so that the method cannot measure any change of $\mu+i\om\eta_{\mu}$ in the direction of $\a$ \cite{Kwon2009,Seo2012}.

To deal with these fundamental drawbacks in the algebraic inversion method, the shear modulus decomposition algorithm based on Helmholtz-Hodge decomposition was developed in \cite{Kwon2009}. This is a much better performing method;  however, it continues to
neglect pressure by using $(\lambda+i\omega\eta_{\lambda})\na\cdot\u=0$, and is thus not realistic.  In \cite{Sinkus2005a,Sinkus2005b}, the curl operator is applied to the elasticity equation (\ref{Eq:viscoelasticity})  to eliminate the troublesome term ($\na\times\na((\lambda+i\omega\eta_{\lambda})\na\cdot\u)=0$). The reconstruction method in \cite{Sinkus2005a,Sinkus2005b} requires third-order derivatives of the noisy data $\u$, making it  very sensitive to noise in the data. A realistic model must  take into account the non-vanishing pressure $p$ \cite{Kozhevnikov1996,Ammari2008}, which can be defined roughly as $p:=\lim_{\lambda\to\infty,~ \na\cdot\u\to 0} ~ (\lambda+i\omega\eta_{\lambda})\na\cdot\u$.

The shear viscoelasticity reconstruction method proposed in this paper is based on the full elasticity model. It does not require any derivative of $\u$. The minimization of a misfit functional involving the discrepancy between the measured and fitted data is considered.
The Fr\'echet derivatives of the functional with respect to $\mu$ and $\eta_{\mu}$ are then computed by introducing an adjoint problem. This Fr\'echet derivatives  based-iterative scheme requires  a well-matched initial guess, because the minimization problem is highly nonlinear and may have multiple local minima. We find a well-matched initial guess that captures the edges of the image of the shear viscoelasticity.

The numerical results presented herein demonstrate
the viability and efficiency of the proposed minimization method.

\section{Reconstruction methods}\label{section-reconstruction}

\subsection{Viscoelastic model}
Let an elastic subject occupy the  smooth domain $\Omega\subset\mathbb{R}^d, d=2,3$  with boundary $\p\Omega$.
To evaluate the viscoelastic tissue properties, we create an internal time-harmonic displacement in the tissue by applying a time-harmonic excitation through the surface of the object.
Under the assumptions of mechanical isotropy and incompressibility in the tissue,  the induced time-harmonic displacement at angular frequency $\omega$, denoted by $\u$, is then governed by the full elasticity equation
\begin{equation}\label{timeharmonic} 2\na\cdot\left((\mu+i\omega\eta_{\mu})\na^s\u\right)+\na((\lambda+i\omega\eta_{\lambda})\na\cdot\u)+\rho\omega^2\u=0 \quad \hbox{in}\,\,\Omega,
\end{equation}
where $\na^s\u=\f{1}{2}(\na\u+\na\u^{t})$ is the strain tensor with $\na\u^{t}$ denoting the transpose of the matrix $\na\u$; $\rho$ is the density of the medium; the complex quantity $\mu+i\omega\eta_{\mu}$ is the shear modulus, with $\mu$ indicating the storage modulus and   $\eta_{\mu}$ indicating the loss modulus reflecting the attenuation of a viscoelastic medium; $\lambda$ and $\eta_{\lambda}$ are the compression modulus and compression viscosity, respectively. We assume that these heterogeneous parameters satisfy \cite{Landau1959}:
\begin{equation*}
\mu>0,\quad \eta_{\mu} > 0,\quad \eta_{\lambda} > 0,\quad d\lambda+2\mu>0.
\end{equation*}\par
We define  the interior domain $\Om'$  and the neighborhood $\mathcal{E}$ of the boundary, $\p\Om$, as
$$\Om':=\{x\in\Om | \mathrm{dist}(x,\p\Om)>\epsilon \, \, \hbox{ with}\,\, \epsilon>0\},\quad \mathcal{E}:=\Om\backslash \overline{\Om'}.$$
We assume that $\mu$ and $\eta_{\mu}$ are known in the region $\mathcal{E}$, and are denoted by $\mu_0$ and $\eta_{{\mu}_0}$, respectively.
We denote by $H^s$ the standard Sobolev space of order $s$ and by $H^s_0$ the closure of $\mathcal{C}^\infty_0$, which is the set of $\mathcal{C}^\infty$ compactly supported functions,  in the $H^s$-norm.

Let $$\widetilde S=\{(\mu, \eta_{\mu}):=(\mu_0, \eta_{{\mu}_0})+(\phi_1,\phi_2)| (\phi_1, \phi_2)\in S\}$$
and
$$\begin{array}{l}
S=\{(\phi_1, \phi_2)\in [H^2_0(\Om)]^2 | c_1<\phi_1+\mu_0<c_2, \\ \hspace{4cm}c_1<\phi_2+\eta_{{\mu}_0}<c_2,\mbox{supp}(\phi_1,\phi_2)\subset\Om'\}.
\end{array}
$$

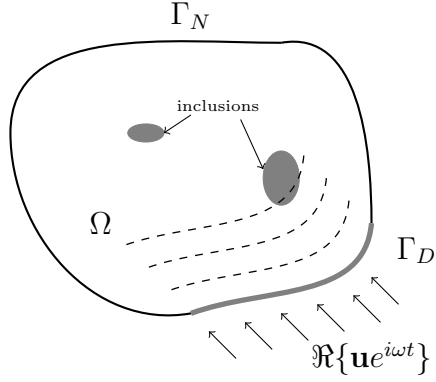
\begin{figure}[h!]
\centering
\begin{tikzpicture}[scale=1.2]
  \draw[thick] (0,0)to[out=-270,in=180](3,1);
  \draw[thick](3,1)to[out=10,in=90](4,-1);
  \draw[thick](2,-2)to[out=-170,in=270](0,0);
  \draw[line width=2pt,gray](4,-1)to[out=-95,in=20](2,-2);
  \draw[draw=gray,fill=gray] (1.5,0)ellipse(0.2 and 0.1);
  \draw[draw=gray,fill=gray] (3,-0.5)ellipse(0.2 and 0.3);
  \node at(2.3,0.3){\tiny inclusions};
  \draw[->](2,0.2)--(1.7,0);
  \draw[->](2.55,0.15)--(2.8,-0.4);
  \node at(1,-1){$\Omega$};
  \node at(4.5,-1.3){$\Gamma_D$};
  \node at(4,-2.5){\color{black}$\Re\{\u e^{i\om t}\}$};
  \node at(2,1.3){$\Gamma_N$};
  \draw[->](2.5,-2.5)--(2.2,-2.2) ;
  \draw[->](2.9,-2.4)--(2.6,-2.1);
  \draw[->](3.3,-2.3)--(3.0,-2);
  \draw[->](3.7,-2.2)--(3.4,-1.9);
  \draw[->](4.1,-2.1)--(3.8,-1.8);
  \draw[->](4.3,-1.9)--(4,-1.6);
  \draw[line width=0.5pt, dashed]
  (3.75,-0.75)to[out=-95,in=20](1.75,-1.75)
  (3.5,-0.5)to[out=-95,in=20](1.5,-1.5)
  (3.25,-0.25)to[out=-95,in=20](1.25,-1.25);
\end{tikzpicture}
\caption{Illustration of the domain and boundary portions.}\label{Fig:domainboundary}
\end{figure}
 Let us take $\overline{\Gamma_D\cup\Gamma_N}=\p\Omega$ and $\Gamma_D\cap\Gamma_N=\emptyset$. Boundary conditions on the displacement field $\u$ are imposed. Typically, we use an acoustic speaker system to generate harmonic vibration.  If the acoustic speaker is placed on the portion $\Gamma_D$ of the boundary $\p\Omega$, then the boundary conditions for $\u$ can be expressed approximately by
\begin{eqnarray*}
&~~~~ \u={\bf g} &\quad\hbox{on}\,\,\Gamma_D, \\
 & 2(\mu+i\omega\eta_{\mu})\na^s\u \, \n+(\lambda+i\omega\eta_{\lambda})(\na\cdot\u) \n=0 &\quad\hbox{on}\,\,{\Gamma_N},
 \end{eqnarray*}
where   $\n$ is the outward unit normal vector to the boundary.

Noting that $\na\cdot\u \approx 0$ (incompressible) whereas   $\lambda = \f{2\mu\nu}{(1-2\nu)}\approx \infty$ (Poisson's ratio $\nu\approx \f 12$) inside the medium,  we introduce the internal pressure $p=(\lambda+i\omega\eta_{\lambda}) \na \cdot \u$, with a limit $p=\lim_{\lambda \to\infty, \na\cdot\u\to 0} (\lambda+i\omega\eta_{\lambda}) \na \cdot \u$. Then, under the limit $\lambda = \infty$ and under the assumption $(\mu, \eta_\mu) \in \widetilde S$, the time harmonic displacement $\u$ and pressure $p$, $(\u, p)\in H^2(\Omega)\times L^2(\Omega)$ satisfy the  following Stokes system \cite{Ammari2013b,Ammari2008}:
\begin{equation}\label{pressureharmonic}
\left\{
\begin{array}{ll}
2\na\cdot\left((\mu+i\omega\eta_{\mu})\na^s\u\right)+\na p+\rho\omega^2\u=0 \quad &\hbox{in} \,\,\Omega,\\
\na\cdot\u=0 \quad  &\hbox{in} \,\,\Omega,\\
\u={\bf g} \quad  &\hbox{on}\,\, \Gamma_D,\\
2(\mu+i\omega\eta_{\mu})\na^s\u \, \n+p \n=0\quad  &\hbox{on} \,\, \Gamma_N.
\end{array}\right.
\end{equation}
Note that if $\Gamma_D = \partial \Omega$ ($\Gamma_N = \emptyset$), then ${\bf g}$ should satisfy the compatibility condition $\int_{\partial \Omega} {\bf g} \cdot \n \, ds = 0$.

Let $\u_m$ denote the displacement data that is measured in $\Om$.  Then, the inverse problem is to reconstruct the distribution of $\mu$ and $\eta_{\mu}$  from the measured data $\u_m$.
\subsection{Optimal control method}

 Define the misfit (or discrepancy) functional $J(\mu, \eta_{\mu})$ in terms  of $\mu$ and $\eta_{\mu}$ by the $L^2$-norm in $\Omega$ of the difference between the numerical solution $\u[\mu,\eta_{\mu}]$ of the forward problem  \eref{pressureharmonic} and the measured displacement data $\u_m=\u_m[\mu^*,\eta_{\mu}^*]$:
\begin{equation}\label{minproblem}
 J(\mu, \eta_{\mu})=\f{1}{2}\int_{\Omega}|\u[\mu,\eta_{\mu}]-\u_m|^2d\x.
\end{equation}
where $\mu^*$ and $\eta_{\mu}^*$ are true distributions of shear elasticity and viscosity, respectively.
The reconstruction of the unknowns $\mu$ and $\eta_{\mu}$  can be obtained by minimizing the misfit functional $J(\mu, \eta_{\mu})$  with respect to $\mu$ and $\eta_{\mu}$.

In order to construct a minimizing sequence of $J(\mu, \eta_{\mu})$, we need to compute  the Fr\'echet derivatives of $ J(\mu, \eta_{\mu})$ with respect to $\mu$ and $\eta_{\mu}$. Assume that $\delta_\mu$ and $\delta_{\eta_{\mu}}$ are small perturbations of $\mu$ and $\eta_\mu$, respectively, by regarding $\f{\delta \mu+i\om \delta_{\eta_{\mu}}}{\mu+i\om \eta_{\mu}}\approx 0$.
For notational simplicity, we denote  $\u_0:=\u[\mu, \eta_{\mu}]$, $p_0:=$ the pressure corresponding to $\u_0$ and $p_0+p_1:=$ the pressure corresponding to  $\u[\mu+\delta_{\mu},\eta_{\mu}+\delta_{\eta_{\mu}}]$. Denoting the perturbation of displacement field by
\begin{equation}\label{Eq:tayloru}
\delta\u:=\u[\mu+\delta_{\mu},\eta_{\mu}+\delta_{\eta_{\mu}}]-\u_0,
\end{equation}
it follows from (\ref{pressureharmonic}) that
\begin{equation}\label{delta_u}
\begin{array}{l}
2\na\cdot\left((\mu+i\omega\eta_{\mu})\na^s\delta\u\right)+\na p_1+\rho\omega^2\delta\u\\
\quad\quad\quad=-2\na\cdot\left((\delta_{\mu}+i\omega\delta_{\eta_{\mu}})\na^s\u_0\right)  -2\na\cdot\left((\delta_{\mu}+i\omega\delta_{\eta_{\mu}})\na^s\delta\u\right)\q\mbox{in }~~\Om.
\end{array}
 \end{equation}

Let $\u_1$ be the solution of the following problem
\begin{equation}\label{perturbedsystem}
\left\{
\begin{array}{ll}
2\na\cdot\left((\mu+i\omega\eta_{\mu})\na^s\u_1\right)+\na p_1+\rho\omega^2\u_1=&\\
\hspace{5cm}-2\na\cdot\left((\delta_{\mu}+i\omega\delta_{\eta_{\mu}})\na^s\u_0\right) ~~&\hbox{in}~\Omega,\\
\na\cdot\u_1=0 \quad &\hbox{in} \,\,\Omega,\\
\u_1={\bf 0} \quad &\hbox{on}\,\, \Gamma_D,\\
2(\mu+i\om\eta_{\mu})\na^s\u_1\n+ p_1\n=0\quad &\hbox{on} \,\,\Gamma_N.
\end{array}\right.
\end{equation}
Now we are ready to state two main theorems in this section which give the Fr\'echet derivatives of $J(\mu, \eta_{\mu})$  with respect to $\mu$ and $\eta_{\mu}$. Denote ${\bf A:B}=\sum_{i,j}A_{ij}B_{ij}$ for two matrices ${\bf A}=(A_{ij})$ and ${\bf B}=(B_{ij})$.
\begin{theorem}\label{Th:frechetderivative}
  For $(\delta \mu + \mu, \delta \eta_\mu + \eta_\mu) \in \widetilde S$,
  if $\u_1$ is defined by (\ref{perturbedsystem}), then we have
\begin{equation}\label{Eq:realfrechet}
\Re\int_{\Omega}\u_1\overline{(\u_0-\u_m)} \, d\x =\Re\int_{\Omega}2(\delta_{\mu}+i\omega\delta_{\eta_{\mu}})\na^s\u_0:\na^s\bar\v \, d\x.
\end{equation}
  Furthermore, the Fr\'echet derivatives of $J(\mu, \eta_{\mu})$  with respect to $\mu$ and $\eta_{\mu}$ are given by
\begin{equation}\label{Eq:frechetderivative}
\f{\p }{\p\mu}J(\mu, \eta_{\mu})=\Re\left[ 2\na^s\u_0:\na^s\bar\v \right],\, \f{\p }{\p\eta_{\mu}}J(\mu, \eta_{\mu})=\Re\left[2(i\omega\na^s\u_0):\na^s\bar\v \right],
\end{equation}
where $\v$ is  the $H^1$ solution of the following adjoint problem \cite{Ammari2010,Evans2010}:
\begin{equation}\label{adjontproblem}
\left\{
\begin{array}{ll}
2\na\cdot\left((\mu-i\omega\eta_{\mu})\na^s\v\right)+\na q+\rho\omega^2\v=({\u_0-\u_m})\quad &\hbox{in}\,\,\Omega,\\
\na\cdot\v=0 \quad &\hbox{in} \,\,\Omega,\\
\v=0\quad &\hbox{on}\,\,\Gamma_D,\\
2(\mu-i\omega\eta_{\mu})\na^s\v\, \n+q\n=0\quad&\hbox{on}\,\,\Gamma_N.
\end{array}\right.
\end{equation}
\end{theorem}

The next theorem shows the differentiability of $J(\mu,\eta_{\mu})$.
\begin{theorem}\label{theorem:high_order}
The misfit functional $J(\mu,\eta_{\mu})$ is Fr\'echet differentiable for $(\mu, \eta_\mu) \in \widetilde{S}$. In other words,
if $\u_1\in H^1(\Om)$ is the weak solution to (\ref{perturbedsystem}), as the perturbations $\delta_{\mu}, \delta_{\eta}\rightarrow0$, we have the following formula:
\begin{equation*}
\hspace{-2.5cm}\left|J(\mu+\delta_{\mu}, \eta_{\mu}+\delta_{\eta_{\mu}})-J(\mu, \eta_{\mu})-\Re\int_{\Omega}\u_1(\overline{\u_0-\u_m})d\x\right|
=O\left((||\delta_{\mu}||_{H^2(\Om)}+||\delta_{\eta_{\mu}}||_{H^2(\Om)})^2\right).
\end{equation*}
\end{theorem}

To prove the  Fr\'echet differentiability Theorem \ref{theorem:high_order} and the  main Theorem \ref{Th:frechetderivative}, we need  the following preliminary results.

 Firstly, we state an interior estimate for the solution of the Stokes system whose proof basically follows from \cite{Chen1998,Giaquinta2012,Li2003} by observing $\nabla \cdot \nabla^s \w = \Delta \w$ for $\w$ satisfying $\na\cdot\w=0$.
\begin{lemma}\label{lemma:estimate}
For ${\bf F}\in L^2(\Om)$ and $(\mu, \eta_{\mu})\in \widetilde S$, let $\w\in H^1(\Om)$ be a weak solution of the following problem:
$$\left\{
\begin{array}{ll}
 2 \na\cdot(\mu+i\om\eta_{\mu})\na^s\w+\na p+\rho\om^2\w={\bf F} \quad&\hbox{in}\,\,\Om,\\
\na\cdot\w=0\quad&\hbox{in}\,\,\Om,\\
\w={\bf 0} \quad &\hbox{on}\,\, \p\Om.
\end{array}\right.
$$
Then, $\w\in H^2(\Om)$ and
\begin{equation}\label{Eq-interior-estimate}
||\w||_{H^2(\Om)}\leq C||{\bf F}||_{L^2(\Om)},\end{equation}
where   $C$ is positive constant independent of ${\bf F}$.
\end{lemma}

The following estimate for $\delta\u$ holds.
\begin{proposition}\label{proposition:delta_u}
The perturbation of displacement field $\delta\u\in H^1(\Om)$ satisfies the following estimate:
$$||\delta\u||_{H^2(\Om)}\leq C(||\delta_{\mu}||_{H^2(\Om)}+||\delta_{\eta_{\mu}}||_{H^2(\Om)})||\u_0||_{H^2(\Om)},$$
where $C$ is positive constant independent of $\delta_\mu$ and $\delta_{\eta_{\mu}}$.
\end{proposition}
\begin{proof}
 From (\ref{delta_u}), $\delta\u$ satisfies
\begin{equation}\label{delta_u2}
\begin{array}{l}
2\na\cdot\left((\mu+\delta_{\mu}+i\omega(\delta_{\eta_{\mu}}+\eta_{\mu}))\na^s\delta\u\right)+\na p_1+\rho\omega^2\delta\u\\
\hspace{5cm}
=-2\na\cdot\left((\delta_{\mu}+i\omega\delta_{\eta_{\mu}})\na^s\u_0\right)\quad\mbox{in }~~\Om.\end{array}
\end{equation}
 Applying the interior estimate  (\ref{Eq-interior-estimate})  to (\ref{delta_u2}) and using H\"older's inequality and Sobolev embedding theorem
  \cite{Adams2003,Evans2010}, we arrive at
$$\begin{array}{ll}
||\delta\u||_{H^2(\Om)}
&\leq C||\na\cdot\left((\delta_{\mu}+i\omega\delta_{\eta_{\mu}})\na^s\u_0\right)||_{L^2(\Om)}\\
&\leq C\left(||\delta_{\mu}+i\omega\delta_{\eta_{\mu}}||_{L^{\infty}(\Om)}||\u_0||_{H^2}+||\na(\delta_{\mu}+i\omega\delta_{\eta_{\mu}})||_{L^4(\Om)}||\na\u_0||_{L^4(\Om)}\right)\\
&\leq C\left(||\delta_{\mu}||_{H^2(\Om)}+||\delta_{\eta_{\mu}}||_{H^2(\Om)}\right)||\u_0||_{H^2(\Om)}.
\end{array}$$
This completes the proof.
\end{proof}

Now we are ready to prove Theorem \ref{theorem:high_order}.
\begin{proof}[Proof of Theorem \ref{theorem:high_order}]
From the definition of $J(\mu, \eta_{\mu})$ in \eref{minproblem}, we have
\begin{equation*}
J(\mu+\delta_{\mu}, \eta_{\mu}+\delta_{\eta_{\mu}})=J(\mu, \eta_{\mu})+\Re\int_{\Omega}\u_1(\overline{\u_0-\u_m})d\x
+\Upsilon,
\end{equation*}
where $\Upsilon$ is
\begin{equation}\label{Eq-Upsilon}
\Upsilon=\Re\int_{\Omega}(\delta\u-\u_1)\cdot(\overline{\u_0-\u_m})d\x
+\f{1}{2}\int_{\Om}|\delta\u|^2d\x.
\end{equation}
Using the adjoint problem (\ref{adjontproblem}),  (\ref{Eq-Upsilon}) can be expressed as
$$\Upsilon=\f{1}{2}\int_\Om|\delta\u|^2 d\x+
\Re\int_{\Omega}(\delta\u-\u_1)\cdot(\overline{2\na\cdot(\mu-i\omega\eta_{\mu})\na^s\v+\na q+\rho\omega^2\v})d\x.
$$
 Using $\na\cdot\delta\u=\na\cdot(\u_0+\delta\u)-\na\cdot\u_0=0$ and homogeneous boundary conditions for $\u_1$ and $\delta\u$, we have
\begin{equation}\label{Eq-Upsilon-Integral}
\Upsilon=\f{1}{2}\int_\Om|\delta\u|^2 d\x-\Re\int_\Om(2\na\cdot(\delta_{\mu}+i\omega\delta_\eta)\na^s\delta\u)\cdot\bar\v d\x.
\end{equation}
Applying H\"older's inequality, $\Upsilon$ is estimated by
$$\begin{array}{ll}
\left|\Upsilon\right|&\leq\f{1}{2}||\delta\u||^2_{L^2(\Om)}+(||\delta_{\mu}||_{L^{\infty}(\Om)}+||\omega\delta_{\eta}||_{L^{\infty}(\Om)})||\na\delta\u||_{L^2(\Om)}||\na\bar\v||_{L^2(\Om)},\\
&\leq C||\na\delta\u||_{L^2(\Om)}\left(\f{1}{2}||\na\delta\u||_{L^2(\Om)}+(||\delta_{\mu}||_{L^{\infty}(\Om)}+||\omega\delta_{\eta}||_{L^{\infty}(\Om)})||\na\bar\v||_{L^2(\Om)}\right).
\end{array}$$
Now we  apply Proposition \ref{proposition:delta_u} to get
\begin{equation*}
\left|\Upsilon\right|\leq C\left(||\delta_{\mu}||_{H^2(\Om)}+||\delta_{\eta_{\mu}}||_{H^2(\Om)}\right)^2\left(||\u_0||_{H^2(\Om)}+||\bar\v||_{H^2(\Om)}\right).
\end{equation*}
The proof is then completed. \end{proof}

Now, it remains to identify the Fr\'echet derivatives of $J(\mu,\eta_{\mu})$. According to Theorem \ref{theorem:high_order},  the Fr\'echet derivatives  $\f{\p }{\p\mu}J(\mu,\eta_{\mu})$ and $\f{\p }{\p\eta_{\mu}}J(\mu,\eta_{\mu})$ can be computed by expressing $\Re\int_{\Omega}\u_1(\overline{\u_0-\u_m})d\x$ in terms of  $\delta_{\mu}$ and $\delta_{\eta_{\mu}}$. These are explained in the proof of Theorem \ref{Th:frechetderivative}.

\begin{proof}[Proof of Theorem \ref{Th:frechetderivative}]

We use the adjoint solution $\v$ in \eref{adjontproblem}  to get
\begin{equation}\label{adjont-use1}
\hspace{-1cm}\int_{\Omega}\u_1\cdot\overline{(\u_0-\u_m)}d\x
=\int_{\Omega}\u_1\cdot(\overline{2\na\cdot\left((\mu-i\omega\eta_{\mu})\na^s\v\right)+\na q+\rho\omega^2\v})d\x.
\end{equation}
Using the vector identity
 $\na\cdot(q\u_1)=\na q\cdot\u_1$ and  divergence free conditions (
$0=\na\cdot\delta\u=\na\cdot\u_1=\na\cdot\v$),  the identity
(\ref{adjont-use1}) can be rewritten as
\begin{equation*}
\hspace{-1cm}\int_{\Omega}\u_1\cdot\overline{(\u_0-\u_m)}d\x
=-\int_{\Omega}2(\mu+i\omega\eta_{\mu})\na^s\u_1:\na^s\bar\v d\x+\int_{\Omega}\rho\omega^2\u_1\cdot\bar\v d\x.
\end{equation*}
Since $\u_1$ satisfies the equation (\ref{perturbedsystem}), we have
$$\begin{array}{ll}
\int_{\Omega}\u_1\cdot\overline{(\u_0-\u_m)}d\x&=\int_{\Omega}[2\na\cdot\left((\mu+i\omega\eta_{\mu})\na^s\u_1\right)+\rho\omega^2\u_1]\cdot\bar\v d\x,\\
&=\int_{\Omega}[-2\na\cdot\left((\delta_{\mu}+i\omega\delta_{\eta_{\mu}})\na^s\u_0\right)+\na p_1]\cdot\bar\v d\x,\\
&=\int_{\Omega}2(\delta_{\mu}+i\omega\delta_{\eta_{\mu}})\na^s\u_0:\na^s\bar\v d\x.
\end{array}$$
This proves the formula (\ref{Eq:realfrechet}).
The formula (\ref{Eq:frechetderivative}) can be obtained directly from  Theorem \ref{theorem:high_order} and  the formula (\ref{Eq:realfrechet}).
This completes the proof. \end{proof}

Based on Theorem \ref{Th:frechetderivative}, the shear modulus and viscosity can be reconstructed by the following  gradient descent iterative scheme:
\begin{description}
  \item[~] {\rm [Step 1] Let $m=0$. Start with an initial guess of shear modulus $\mu^0$ and shear viscosity $\eta_\mu^0$.}
\item[~] {\rm [Step 2] For $m=0,1,\cdots$, compute $\u_0^m$ by solving the forward problem \eref{pressureharmonic} with $\mu$ and $\eta_\mu$ replaced by $\mu^m$ and $\eta_\mu^m$, respectively. Compute $\v^m$ by solving the adjoint problem \eref{adjontproblem} with $\mu, \eta_\mu, \u_0$ replaced by $\mu^m, \eta_\mu^m, \u_0^m$, respectively.}
  \item[~] {\rm  [Step 3]  For $m=0,1,\cdots$, compute the Fr\'echet derivatives $\f{\p J}{\p\mu}(\mu^m,\eta_\mu^m)$ and $\f{\p J}{\p\eta_{\mu}}(\mu^m,\eta_{\mu}^{m})$.}
\item[~]  {\rm  [Step 4] Update $\mu$ and $\eta_\mu$ as follows:}
\begin{equation}\label{Eq:iterationscheme}
\left\{
\begin{array}{ll}
\mu^{m+1}&=\mu^{m}-\delta\f{\p J}{\p\mu}(\mu^m, \eta_\mu^m),\\
\eta_{\mu}^{m+1}&=\eta_{\mu}^{m}-\delta\f{\p J}{\p\eta_{\mu}}(\mu^m, \eta_{\mu}^{m}).
\end{array}\right.
\end{equation}
\item[~]{\rm  [Step 5] Repeat  Steps $2,3$, and $4$ until $||\mu^{m+1}-\mu^m||\leq\epsilon$ and $||\eta_\mu^{m+1}-\eta_\mu^m||\leq\epsilon$ for a given $\epsilon>0$.}
\end{description}

\subsection{Initial guess}\label{section-initial-guess}
Numerous simulations show that the reconstruction from an adjoint-based optimization method may  converge to some local minimum that is  very different from the true solution when the initial guess is far from the true solution. We  observed that different initial guesses produce different reconstructions, and thus a good initial guess is necessary for accurate reconstruction using the iterative method (\ref{Eq:iterationscheme}).

We examine the optimization method using the initial guess obtained by the direct inversion method (\ref{Eq:directinversion}). Numerical simulations with this initial guess showed that serious reconstruction errors occur near the interfaces of different materials in the same domain; the direct inversion method cannot probe those interfaces. We found empirically that it is important to find an initial guess capturing the interfaces of different materials for the effective use of the optimization method.

To develop a method of finding such a good initial guess, we adopt the hybrid one-step method \cite{Lee2010} which consider the following simplified model  ignoring the pressure term:
\begin{equation}\label{Initial1}
 2\na\cdot (\mu+i\om\eta_{\mu})\na^s \u^\diamond+\rho\om^2 \u^\diamond~=~0\quad\mbox{in}\,\,\Om,
 \end{equation}
where $\u^\diamond$ is regarded as a good approximation of $\u[\mu,\eta_{\mu}]$.
To probe the discontinuity of $ (\mu+i\om\eta_{\mu})\na^s \u^\diamond$, we apply the Helmholtz decomposition
 \begin{equation}\label{Eq:helmholtz}
 (\mu+i\om\eta_{\mu})\na^s \u^\diamond=\na {\bf f}+\na\times{\bf W} \,\,\mbox{with}\,\,\na\cdot{\bf W}={\bf 0},
 \end{equation}
 where ${\bf f}$ and ${\bf W}$ are vector and matrix, respectively. The curl of matrix is defined in column-wise sense: $\na\times{\bf W}=\na \times(W_1, W_2, W_3)=(\na\times W_1,\na\times W_2, \na\times W_3)$, where $W_j$ is the $j$-th column of matrix ${\bf W}$ for $j=1,2,3$.
Taking dot product of (\ref{Eq:helmholtz}) with $\na^s\u^\diamond$ gives the following formula
 \begin{equation}\label{Eq:initial}
 \mu+i\om\eta_{\mu}=\f{\na {\bf f}:\na^s\bar \u^\diamond}{|\na^s \u^\diamond|^2}+\f{\na\times {\bf W}:\na^s \bar \u^\diamond}{|\na^s \u^\diamond|^2}.
 \end{equation}
By taking the divergence to the equation \eref{Eq:helmholtz}, we have
 \begin{equation}\label{Eq:helmholtz1}
\Delta {\bf f} = - \frac{1}{2} \rho\om^2 \u^\diamond \quad \mbox{in } \Omega.
 \end{equation}
 By taking the curl operation to the equation \eref{Eq:helmholtz}, we have
 \begin{equation}\label{Eq:helmholtz2}
\Delta {\bf W} = \nabla \times ((\mu+i\om\eta_{\mu}) \nabla^s \u^\diamond )
\quad \mbox{in } \Omega. \end{equation}

 Our proposed method for determining the initial guess is based on the modifying of hybrid one-step method.
Using (\ref{Eq:helmholtz1}), an approximation of the vector potential ${\bf f}$ corresponding to the measurement $\u_m$ can be computed by
\begin{equation}\label{Eq:f-tilde}
\left\{
\begin{array}{ll}
\Delta {\bf\widetilde f} = - \frac{1}{2} \rho\om^2 \u_m \quad &\mbox{in } \Omega,\\
\nabla {\bf \widetilde f} \, \n =  (\mu_0+i\om\eta_{\mu_0}) \nabla^s \u_m \, \n  \quad &\mbox{on } \partial \Omega.
\end{array}\right.
\end{equation}
On the other hand,   ${\bf W}$ can not be computed  directly from $\u_m$  since (\ref{Eq:helmholtz2}) contains unknown terms  $\mu$ and $\eta_{\mu}$.
Regarding $\mu+i\om\eta_{\mu}$  in (\ref{Eq:helmholtz2}) as  $\f{\na {\bf\widetilde f}:\na^s\bar \u_m}{|\na^s \u_m|^2}$  (see \eref{Eq:initial}), we can compute a rough approximation of ${\bf W}$ by solving
\begin{equation}\label{Eq-W1}
 \left\{
 \begin{array}{ll}
\Delta {\bf W}_1 = \nabla \times (\f{\na {\bf\widetilde f}:\na^s\bar \u_m}{|\na^s \u_m|^2} \nabla^s \u_m)
\quad &\mbox{in } \Omega,\\
{\bf W}_1 = {\bf 0}
\quad &\mbox{on } \partial \Omega.
\end{array}\right.
\end{equation}
Similarly,  approximating $\mu+i\om\eta_{\mu}$  by direct inversion formula \eref{Eq:directinversion}, we can compute ${\bf W}$ by solving
\begin{equation}\label{Eq-W2}
\left\{
\begin{array}{ll}
\Delta {\bf W}_2 = \nabla \times (-\f{\rho\omega^2 (\a\cdot \u_m)}{\na\cdot\na (\a\cdot\u_m)} \nabla^s \u_m)
\quad &\mbox{in } \Omega,\\
{\bf W}_2 = {\bf 0}
\quad &\mbox{on } \partial \Omega,
\end{array}\right.
\end{equation}
where ${\bf a}$ is any nonzero vector.

Now, we use the formula  \eref{Eq:initial} to get  the initial guess of shear modulus by substituting ${\bf f=\tilde f}$, ${\bf W}=({\bf W}_1+{\bf W}_2)/2$ and $\u^\diamond=\u_m$:
 \begin{equation}\label{Eq:initial2}
 \mu^0+i\om\eta_{\mu}^0=\f{\na {\bf\widetilde f}:\na^s\bar \u_m}{|\na^s \u_m|^2}+\f{\na\times ({\bf W}_1+{\bf W}_2):\na^s \bar \u_m}{2|\na^s \u_m|^2}.
 \end{equation}
 In  formula \eref{Eq:initial2}, the first term provides information in the wave propagation direction while the second term gives the information in the tangent direction of the wave propagation as shown in \cite{Lee2010}.
Note that if this initial guess is not satisfactory for the adjoint-based optimization problem, one can update the initial guess formula to obtain more accurate one by replacing $(\mu+i\om\eta_{\mu})$ in (\ref{Eq:helmholtz2}) by (\ref{Eq:initial2}).

Numerical experiments demonstrates the possibility of probing the discontinuity of the shear modulus effectively. We emphasize that the initial guess plays an important role in  Newton's iterative reconstruction algorithm based on the adjoint approach. By observing the adjoint problem (\ref{adjontproblem}), the load term $\u_0-\u_m$ is related to the measured data and the initial guess in the first iteration step. If the initial guess ensure that $||\u_0-\u_m||$ is small in certain norm, the iteration scheme will converge and give good results. Otherwise, the initial guess makes $||\u_0-\u_m||$ far  from 0 in certain norm, and the iteration scheme may not converge. This will be  discussed in section \ref{section-numerical}.

\subsection{Local reconstruction}\label{Section-local}
In MRE, the time-harmonic displacement, $\u_m$, in the tissue is measured via  phase-contrast-based MR imaging. Hence, the signal-to-noise ratio (SNR) of the data is related to that of the MR phase images, which varies from one region to another. For example, the SNR of data $\u_m$ is very low in MR-defected regions, including the lungs, outer layers of bones, and some gas-filled organs. When the domain, $\Om$, contains such defected regions, the reconstructed image qualities may be  seriously degraded by locally low SNR data in the defected regions. As a result, it would be desirable to exclude defected regions from $\Om$ to prevent errors spreading in the image reconstruction.

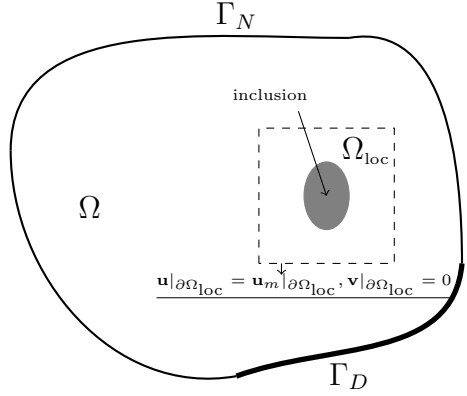
\begin{figure}[h!]
\centering
\begin{tikzpicture}[scale=1.5]
  \draw[dashed](2.2,-1)rectangle(3.4,0.2);
  \node at(3.15, 0){$\Omega_{\hbox{\tiny loc}}$};
  \draw[thick] (0,0)to[out=-270,in=180](3,1);
  \draw[thick](3,1)to[out=10,in=90](4,-1);
  \draw[thick](2,-2)to[out=-170,in=270](0,0);
  \draw[line width=2pt](4,-1)to[out=-95,in=20](2,-2);
  \draw[draw=gray,fill=gray] (2.8,-0.4)ellipse(0.2 and 0.3);
  \node at(2.3,0.5){\tiny inclusion};
  \draw[->](2.55,0.35)--(2.8,-0.4);
  \node at(0.7,-0.5){$\Omega$};
  \node at(3,-2){$\Gamma_D$};
  \node at(2,1.2){$\Gamma_N$};
  \node at(2.6,-1.2){\tiny $\underline{{\bf u}|_{\p\Omega_{\hbox{\tiny loc}}}={\bf u}_m|_{\p\Omega_{\hbox{\tiny loc}}},{\bf v}|_{\p\Omega_{\hbox{\tiny loc}}}=0}$};
  \draw[->](2.4,-1)--(2.4,-1.1);
\end{tikzpicture}
\caption{Illustration of the localization of the small anomaly in certain subdomain.}\label{Fig:localization_1}
\end{figure}

The proposed method is capable of a local reconstruction by restricting to a local domain of the interest. To be precise, let  $\Omega_{\mbox{\tiny loc}}$ be a subdomain of $\Om$ in which  $\u_m$ has high SNR. Then, we consider the localized minimization problem
\begin{equation}\label{minproblem-loc}
 J_{\mbox{\tiny loc}}(\mu, \eta_{\mu})=\f{1}{2}\int_{\Omega_{\mbox{\tiny loc}}}|\u_{\mbox{\tiny loc}}[\mu,\eta_{\mu}]-\u_m|^2d\x \end{equation}
with $\u_{\mbox{\tiny loc}}[\mu,\eta_{\mu}]$ being the solution of
\begin{equation}\label{localdomain}
\left\{
\begin{array}{ll}
2\na\cdot\left((\mu+i\omega\eta_{\mu})\na^s\u\right)+\na p+\rho\omega^2\u=0  \quad&\hbox{in}\,\,\Omega_{\mbox{\tiny loc}},\\
\na\cdot\u=0\quad&\hbox{in}\,\,\Omega_{\mbox{\tiny loc}},\\
\u=\u_m \quad&\hbox{on}\,\,\p\Omega_{\mbox{\tiny loc}}.
\end{array}\right.
\end{equation}
As before, we need to compute the corresponding adjoint problem to get Fr\'echet derivative:
\begin{equation}\label{localadjont}
\left\{
\begin{array}{ll}
2\na\cdot\left((\mu-i\omega\eta_{\mu})\na^s\v\right)+\na q+\rho\omega^2\v=\u_{0,\mbox{\tiny loc}}-\u_m\quad &\hbox{in}\,\,\Omega_{\mbox{\tiny loc}},\\
\na\cdot\v=0 \quad&\hbox{in}\,\,\Omega_{\mbox{\tiny loc}},\\
\v=0\quad &\hbox{on}\,\,\p\Omega_{\mbox{\tiny loc}}.
\end{array}\right.
\end{equation}
There is no difference between the local reconstruction in $\Omega_{\mbox{\tiny loc}}$ and  the global reconstruction with $\Om$, except the boundary conditions. As in \eref{Eq:iterationscheme}, the local reconstruction can be done by solving \eref{localdomain} and \eref{localadjont} with the initial guess \eref{Eq:initial2}.   Local reconstruction requires that neither the boundary conditions need to be used  on the whole domain, $\Omega$,  nor that the exact shape of $\Omega$ needs to be known. Numerical simulations verify  the effectiveness of this local reconstruction, and  will be discussed in section \ref{section-numerical}.

\section{Numerical simulations}\label{section-numerical}
In this section, we perform several numerical experiments to illustrate the effectiveness of the shear viscoelasticity reconstruction algorithm  proposed in the previous section.

To implement the reconstruction algorithm (\ref{Eq:iterationscheme}) proposed in section \ref{section-reconstruction}, we
use the algorithm (\ref{Eq:initial}) in section \ref{section-initial-guess} to initialize the iteration scheme.  For numerical experiments, we set the two dimensional domain as $\Om =[0,10] \times [0,10]$ cm$^2$ with a boundary denoted by $\p\Om=\Gamma_{D}\cup\Gamma_{N}$; see figure \ref{Fig-displacement-model1} (a).
We apply the FEM method in Matlab (MathWorks In.) to solve the forward problem (\ref{pressureharmonic}) as well as the adjoint problem (\ref{adjontproblem}) at each iteration step in the algorithm (\ref{Eq:iterationscheme}).

We set three different types of shear viscoelasticity distribution which are shown in the first column of figure \ref{Fig-reconstruction-model123} along with the true distribution of shear modulus and shear viscosity. The first and second rows are model 1, the third and fourth rows are model 2, and the fifth and sixth rows are model 3. For each model, the upper row shows elasticity while the lower row shows viscosity. Our numerical experiments are based on these three models. We generate two dimensional displacements $\u_m=(u_1, u_2)^t$ by solving the problem (\ref{pressureharmonic}) with frequency $\f{\om}{2\pi}$=70Hz and density $\rho=1g\cdot cm^{-2}$. We apply the vibration to $\Gamma_{D}$, and  the other three sides boundaries are set to be traction free:
 \begin{equation}\label{pressureharmonic}
\left\{
\begin{array}{ll}
\u=(0.3, 0.3) \quad &\hbox{on}\,\, \Gamma_D,\\
2(\mu+i\omega\eta_{\mu})\na^s\u \, \n+p \n=0\quad &\hbox{on} \,\, \Gamma_N.
\end{array}\right.
\end{equation}
 For example,   model 1 has the displacement fields  shown in figure \ref{Fig-displacement-model1} where (b) and (c) are real parts of $u_1$ and $u_2$, and (d) and (e) are imaginary parts of $u_1$ and $u_2$, respectively.

\begin{figure}[!h]
\centering
\setlength{\tabcolsep}{2pt}
\begin{tabular}{ccccc}
\includegraphics[height=2.25cm]{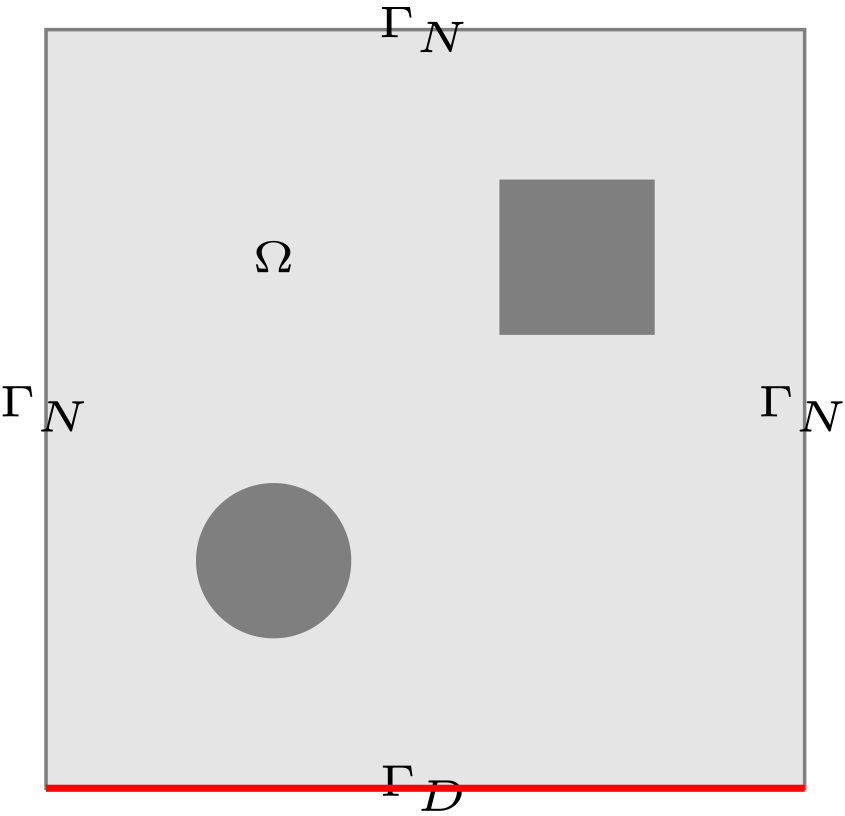}&
\includegraphics[height=2.2cm]{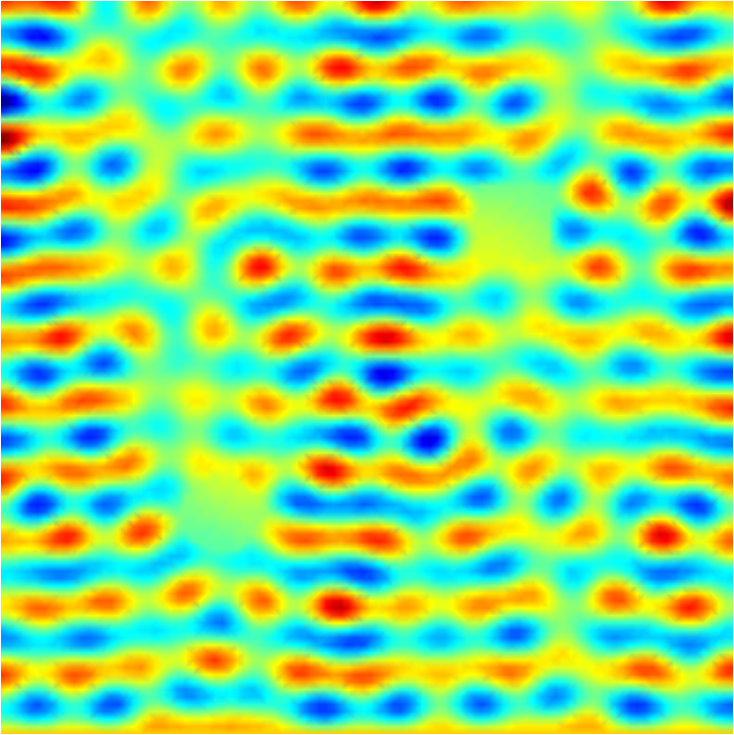}&
\includegraphics[height=2.2cm]{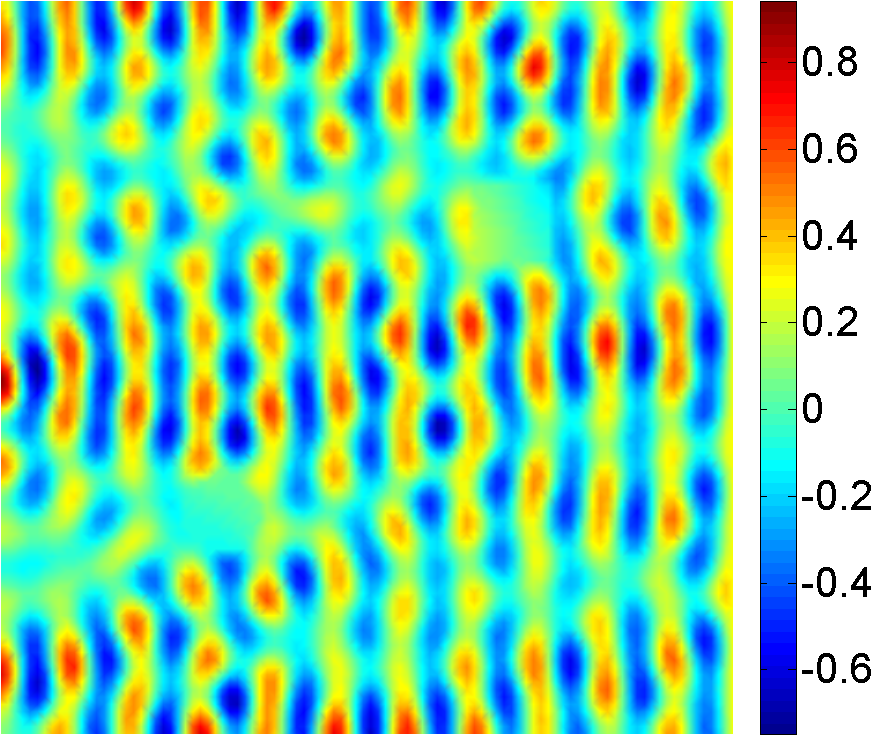}&
\includegraphics[height=2.2cm]{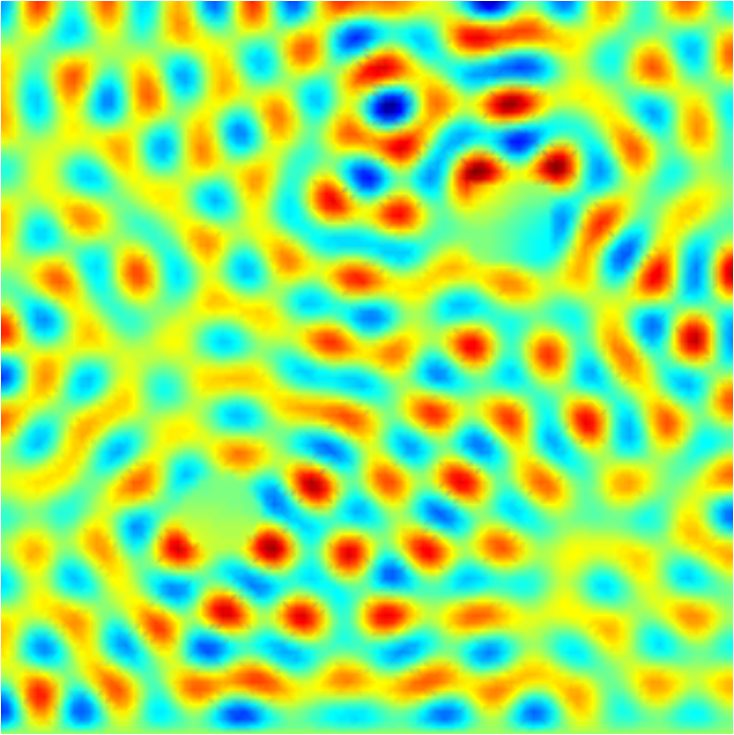}&
\includegraphics[height=2.2cm]{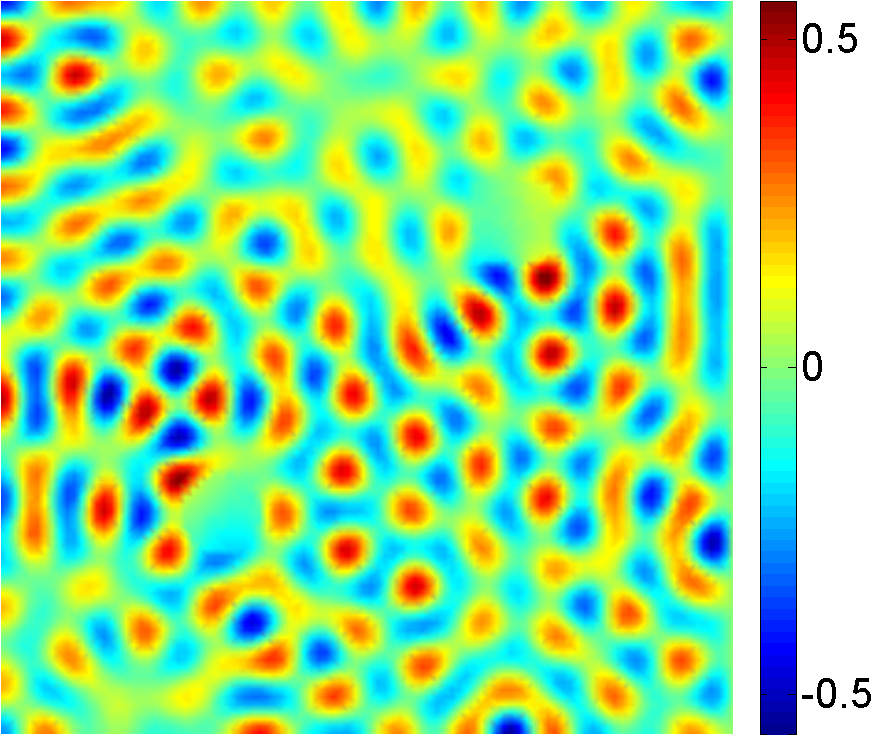}\\
(a)&(b)&  (c)& (d) & (e)
\end{tabular}
\caption{Model 1 and the displacement fields. (a) Model 1; (b) and (c) are real parts of $u_1$ and $u_2$; (e) and (f) are imaginary parts of them, respectively.}
\label{Fig-displacement-model1}
\end{figure}
The next step is to implement our algorithm making use of these displacement fields with certain initial guesses of the distribution of viscoelasticity.
We generate the initial guess by the direct inversion method (\ref{Eq:directinversion}) shown in the third column of  figure \ref{Fig-reconstruction-model123} and the hybrid one-step method (\ref{Eq:initial}) shown in the fifth column of figure \ref{Fig-reconstruction-model123}. From the generated initial guess, we can see that the reconstruction by the hybrid one-step method is much better than that of the direct inversion method in catching the inhomogeneous property of the medium. We have already explained the underlying mathematical reason for this phenomenon. We  use the initial guesses from these two methods to initialize our proposed method, and the corresponding numerical results for each model are shown in the fourth column and last column of  figure \ref{Fig-reconstruction-model123}, respectively. For comparison, we also show the reconstruction with a homogeneous initial guess in each second column of figure  \ref{Fig-reconstruction-model123}.

The reconstruction results (see figure  \ref{Fig-reconstruction-model123}) show that the proposed method can reconstruct the viscoelasticity distribution with high accuracy (see (f) column) using a well-matched initial guess ( see (e) column). Otherwise, poor initial guesses (for example, the homogeneous initial guess and (c)), leads to unsatisfactory reconstructed images (see (b) and (d) columns).

\begin{figure}[!h]
\centering
\setlength{\tabcolsep}{2pt}
\begin{tabular}{cccccc}
\includegraphics[height=2cm]{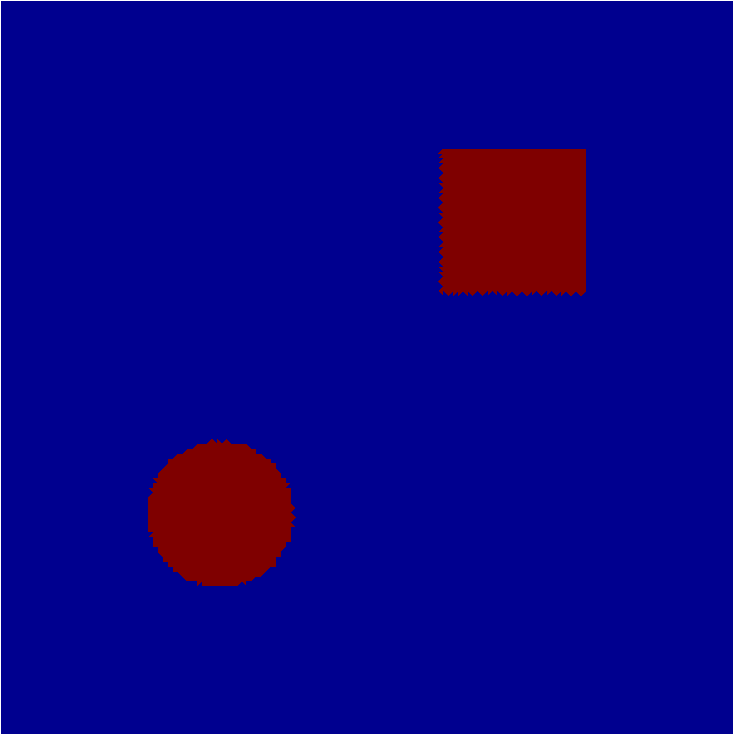}&
\includegraphics[height=2cm]{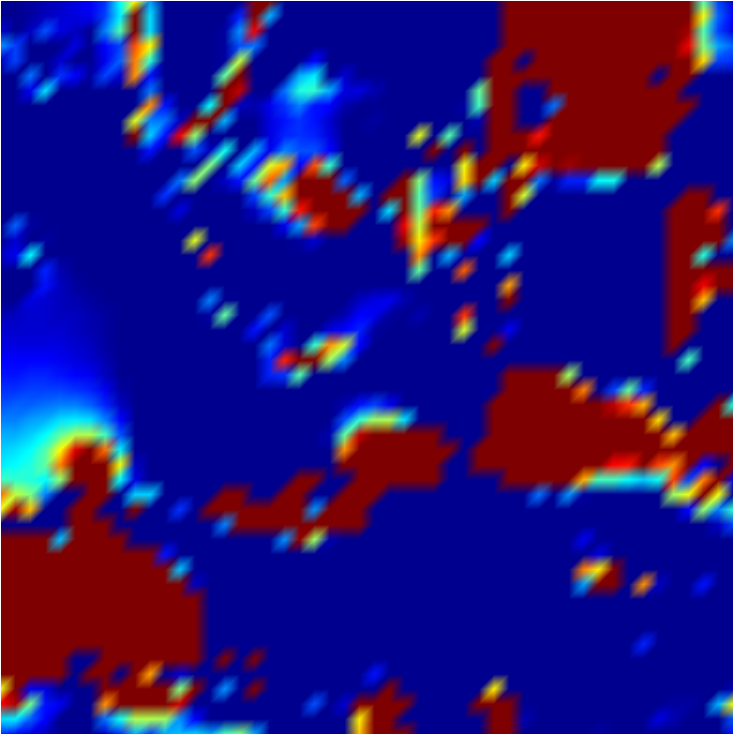}&
\includegraphics[height=2cm]{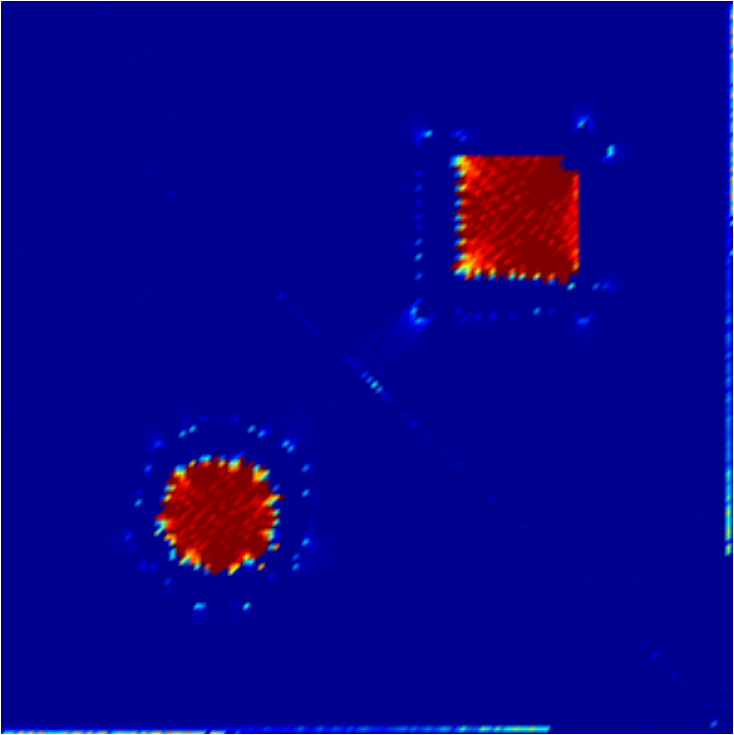}&
\includegraphics[height=2cm]{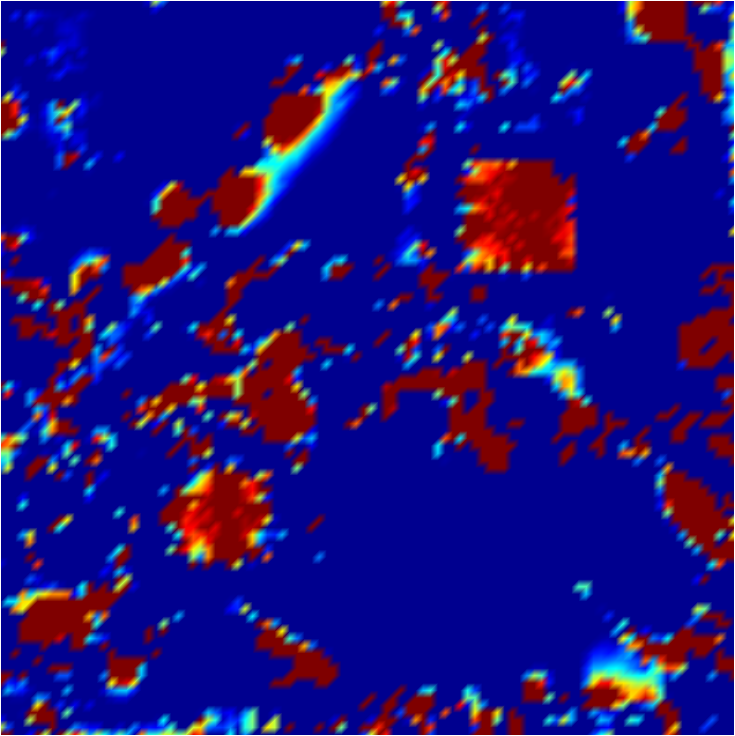}&
\includegraphics[height=2cm]{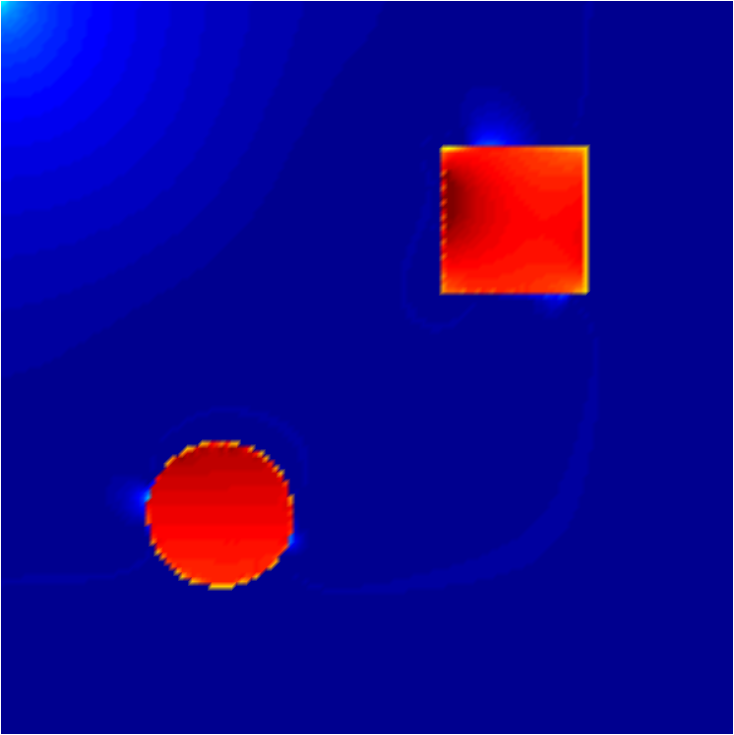}&
\includegraphics[height=2cm]{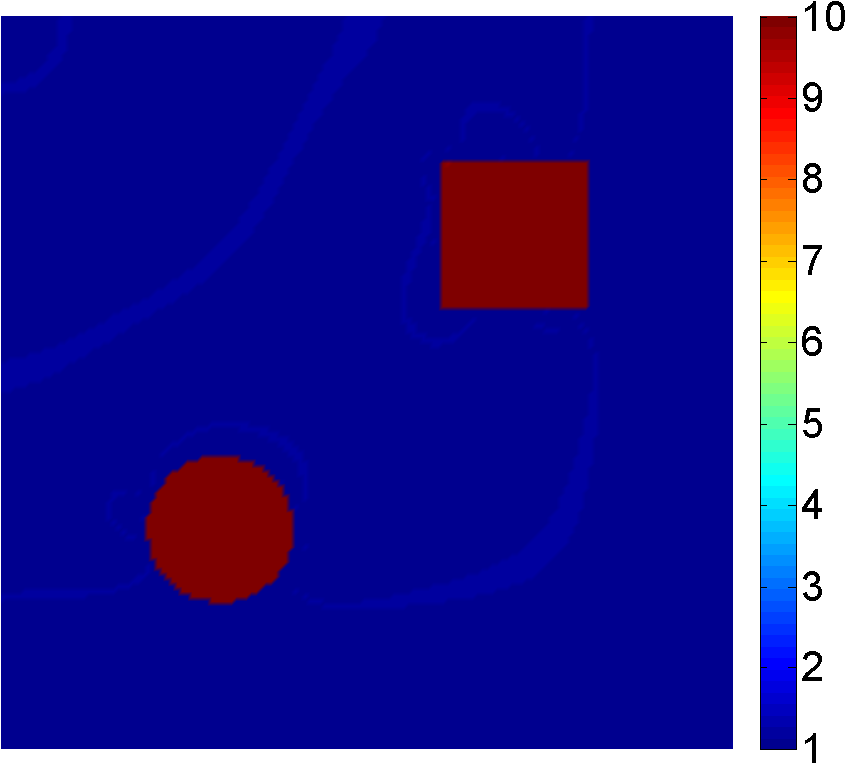}\\
\includegraphics[height=2cm]{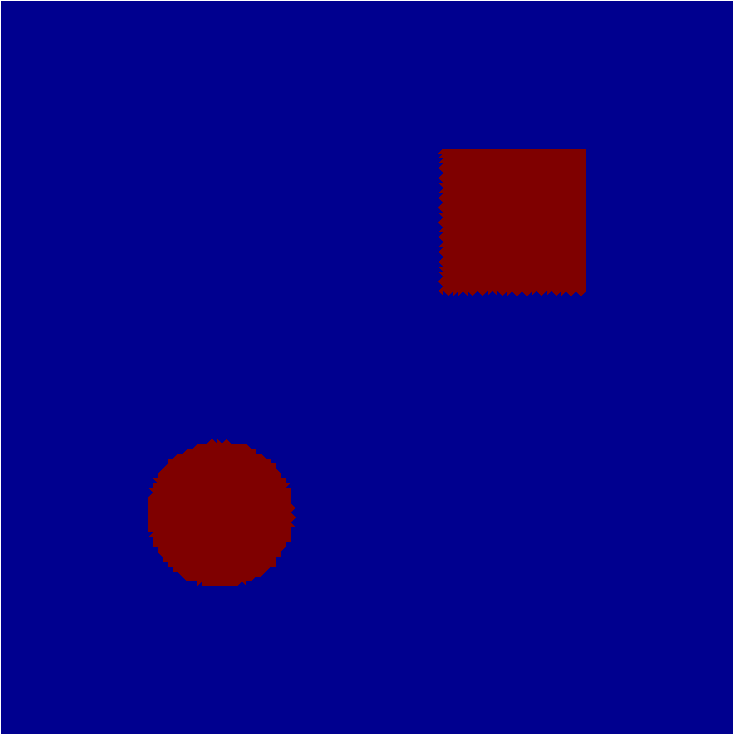}&
\includegraphics[height=2cm]{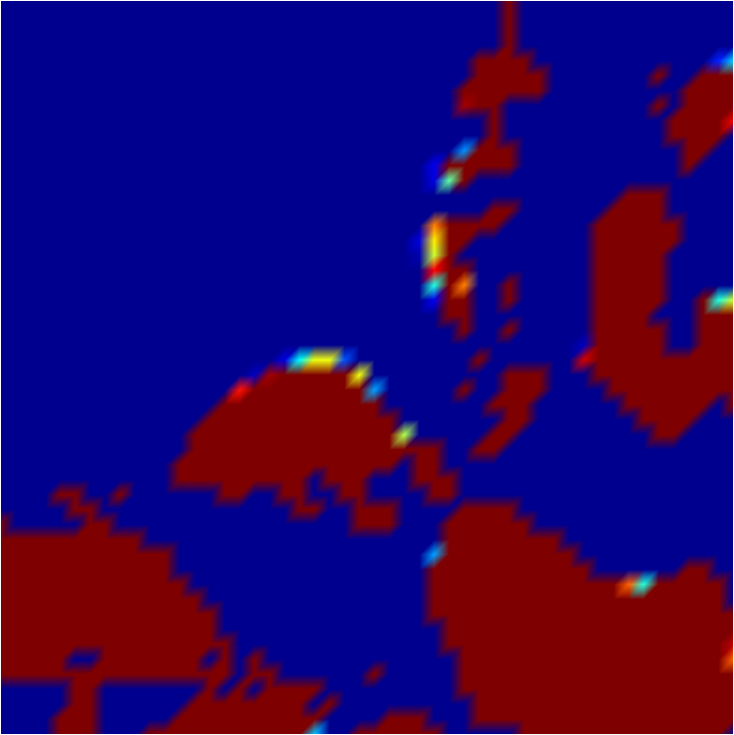}&
\includegraphics[height=2cm]{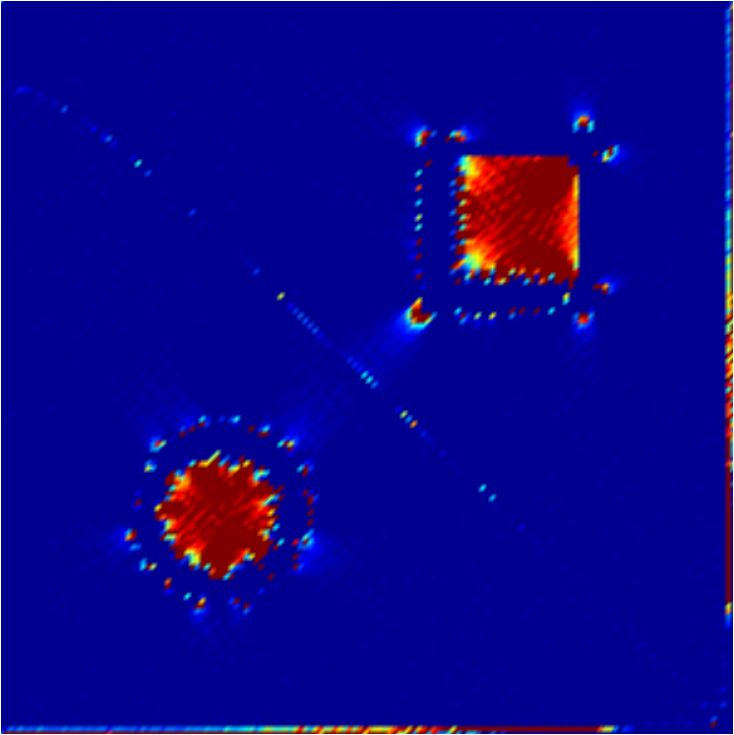}&
\includegraphics[height=2cm]{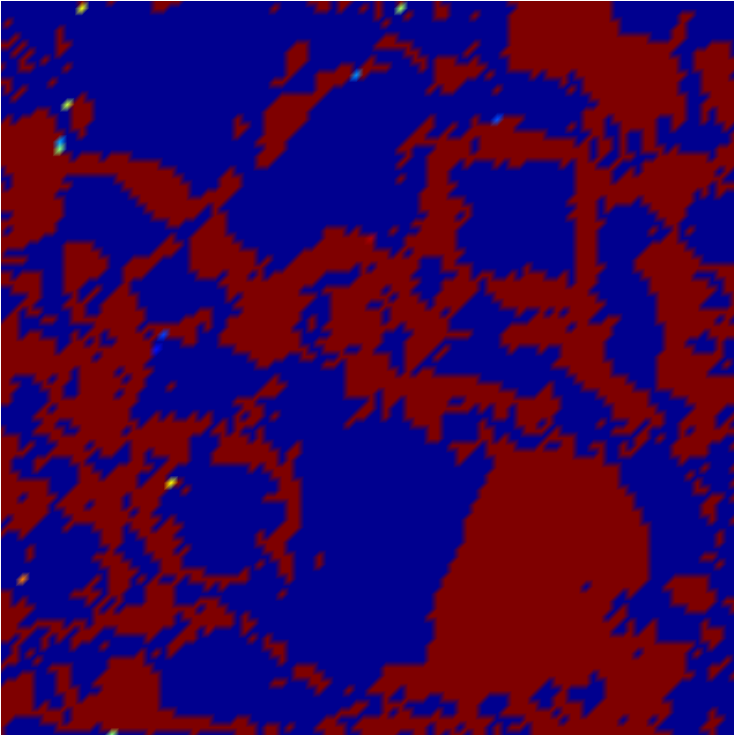}&
\includegraphics[height=2cm]{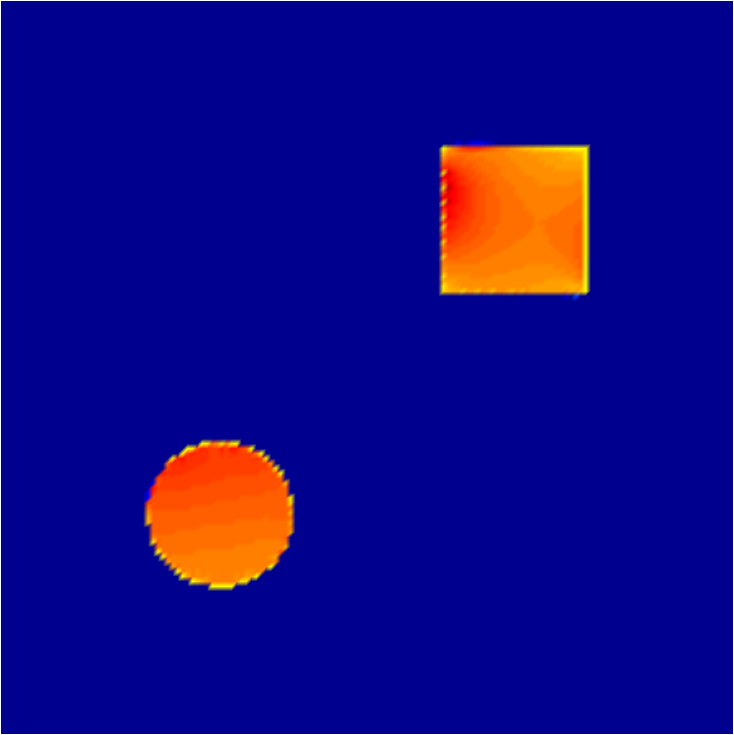}&
\includegraphics[height=2.1cm]{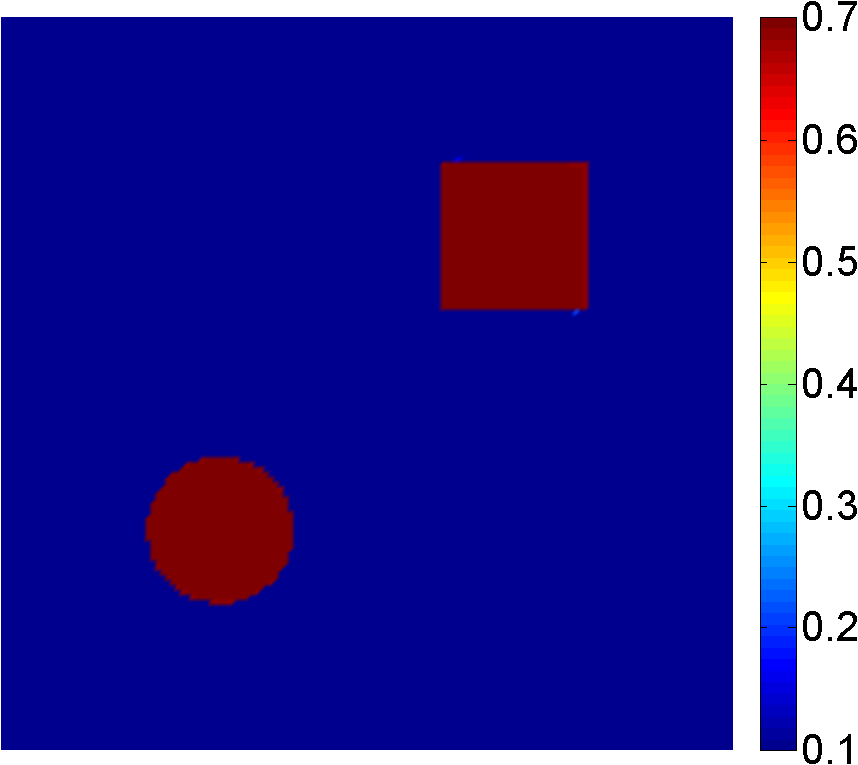}\\
\includegraphics[height=2cm]{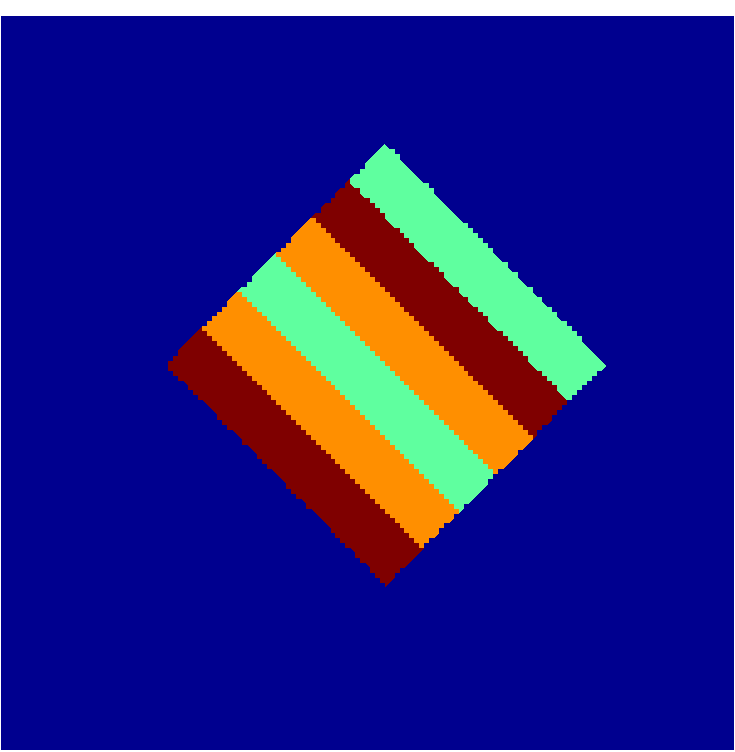}&
\includegraphics[height=1.95cm]{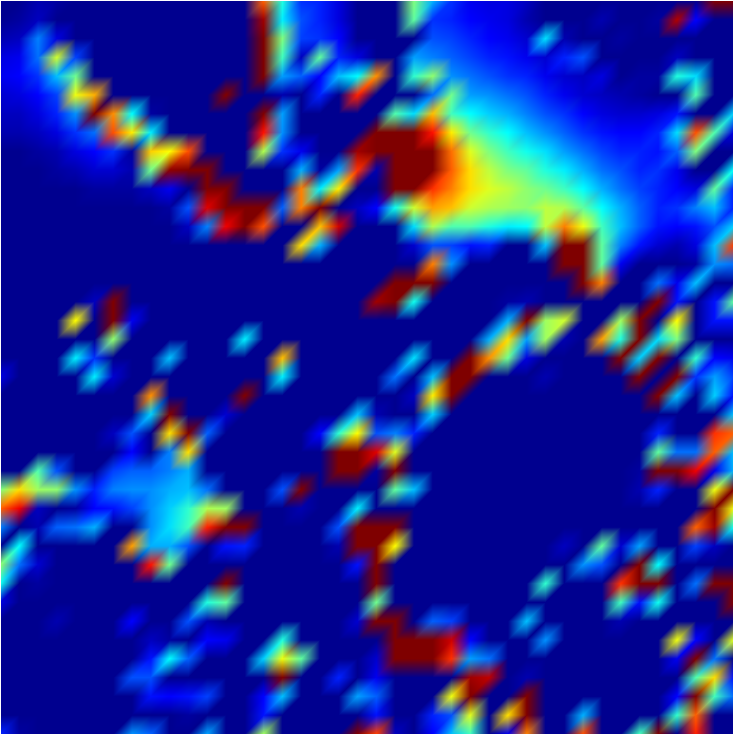}&
\includegraphics[height=2cm]{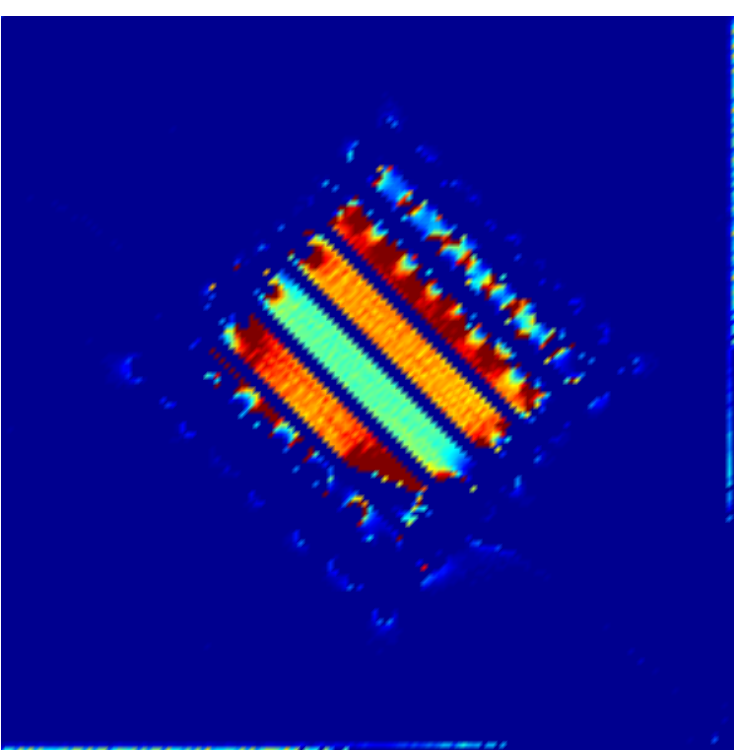}&
\includegraphics[height=1.95cm]{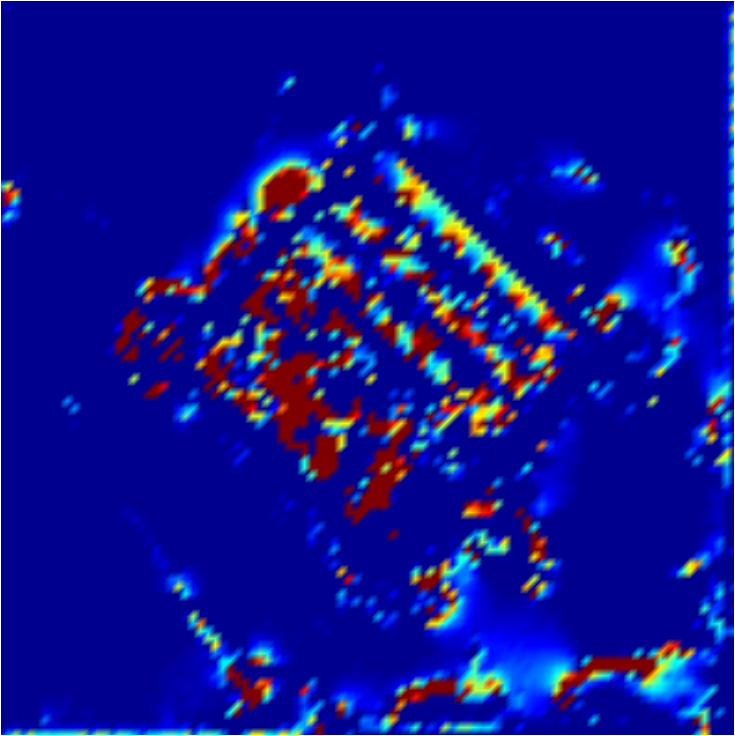}&
\includegraphics[height=2cm]{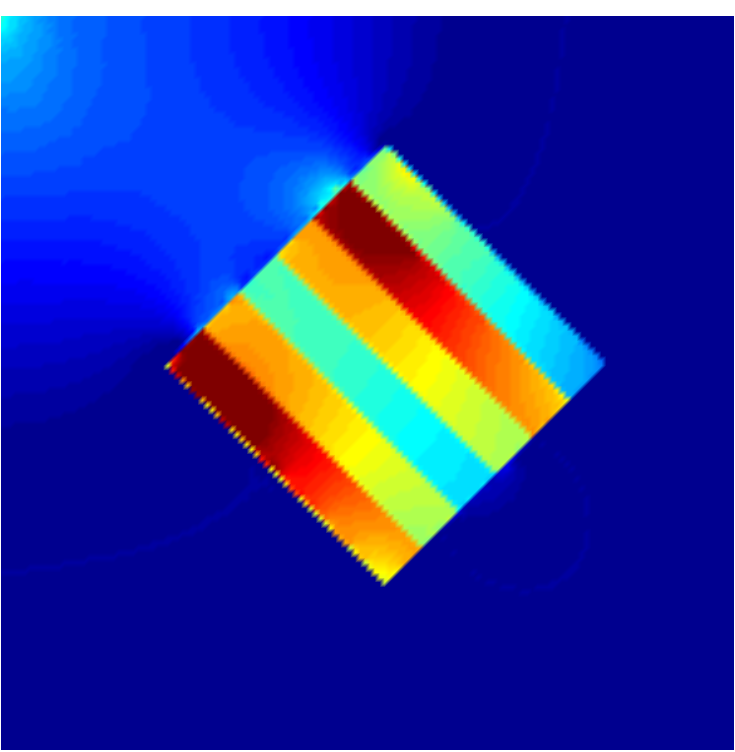}&
\includegraphics[height=2cm]{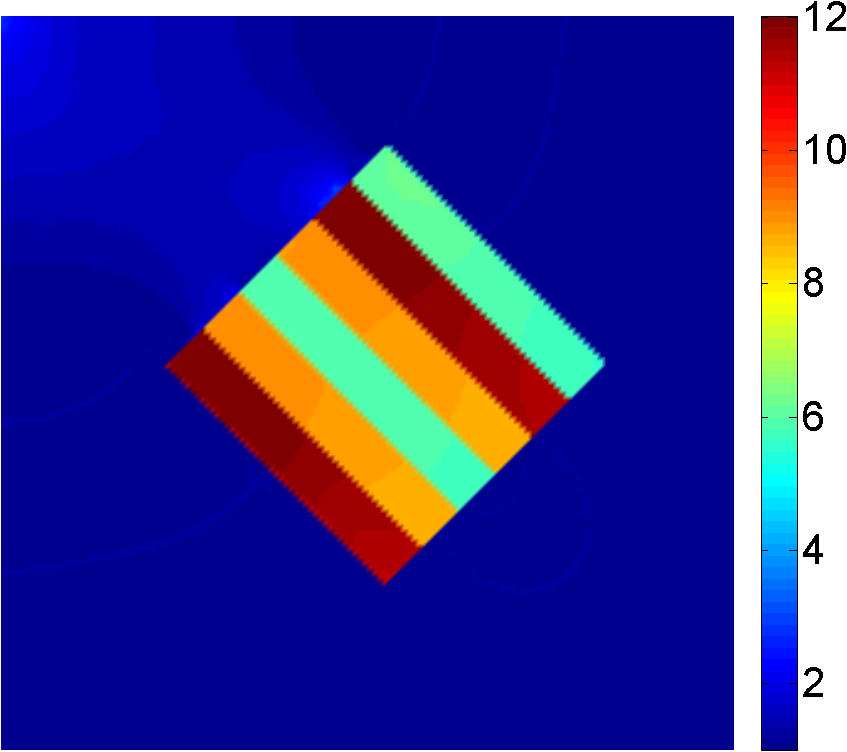}\\
\includegraphics[height=2cm]{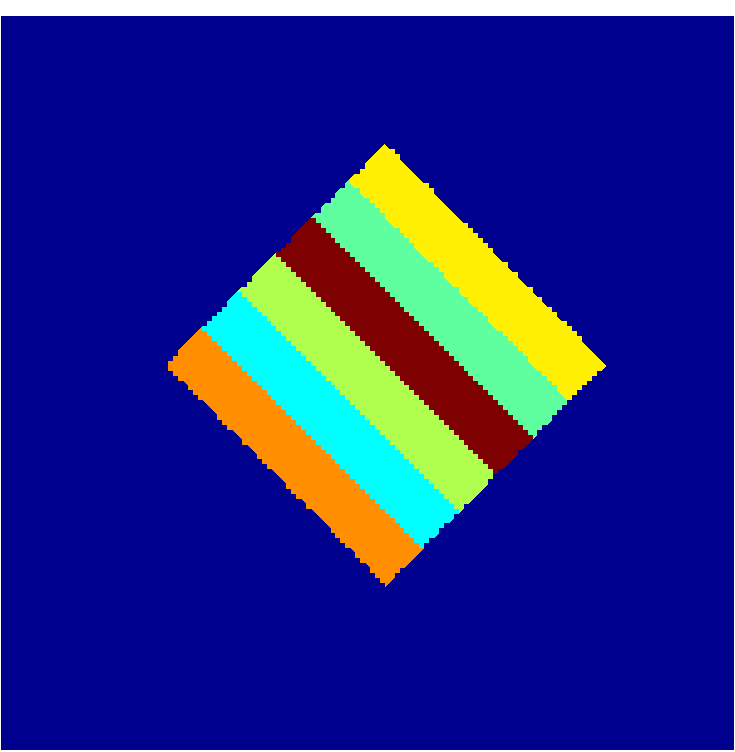}&
\includegraphics[height=1.95cm]{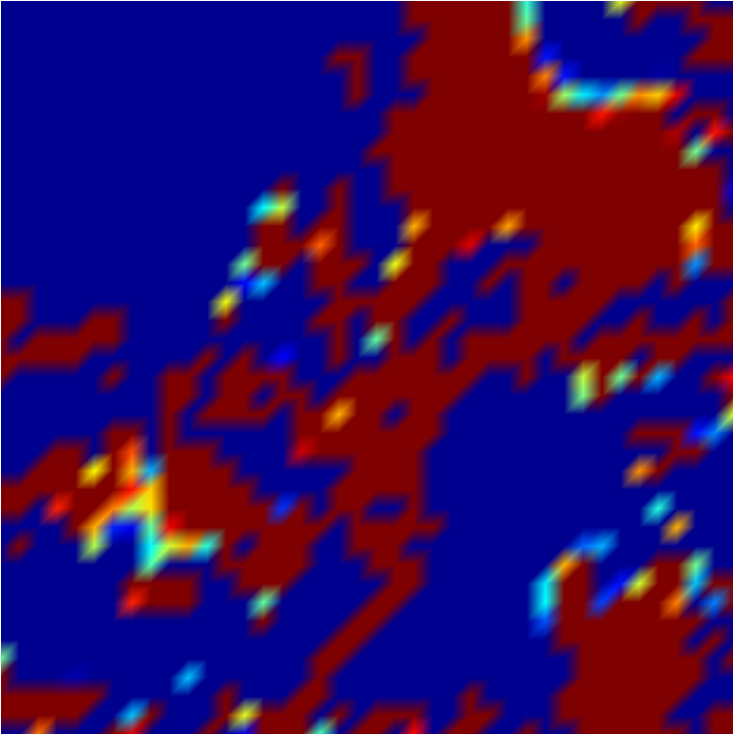}&
\includegraphics[height=2cm]{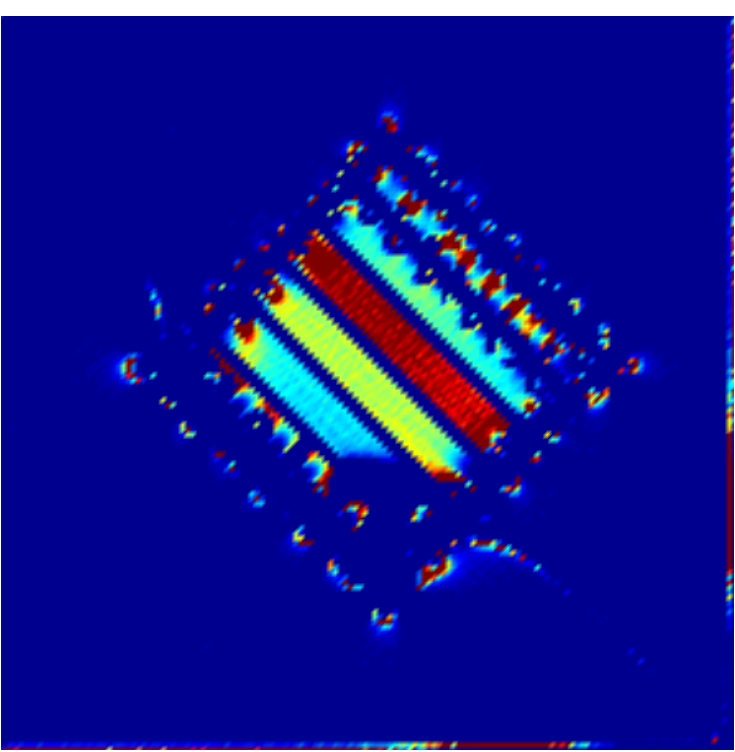}&
\includegraphics[height=1.95cm]{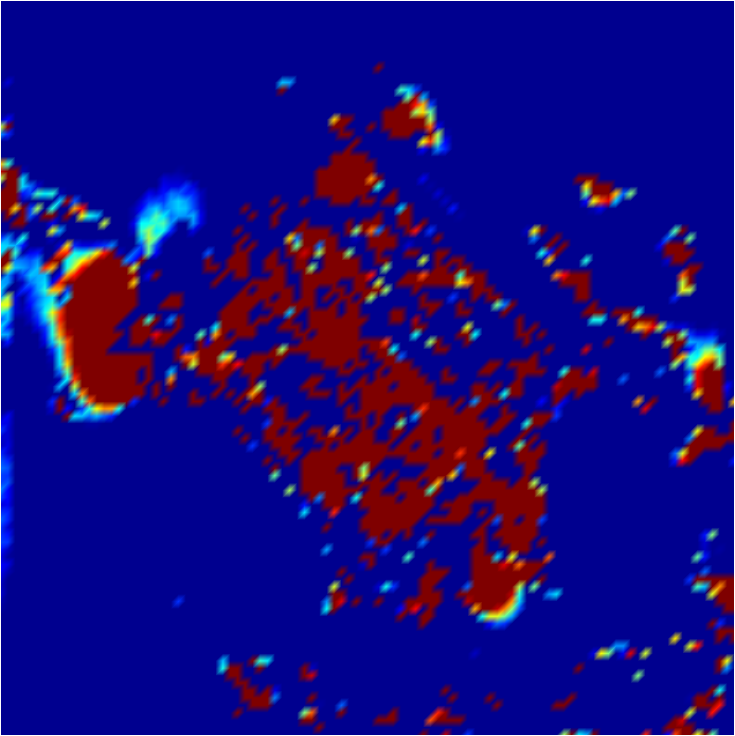}&
\includegraphics[height=2cm]{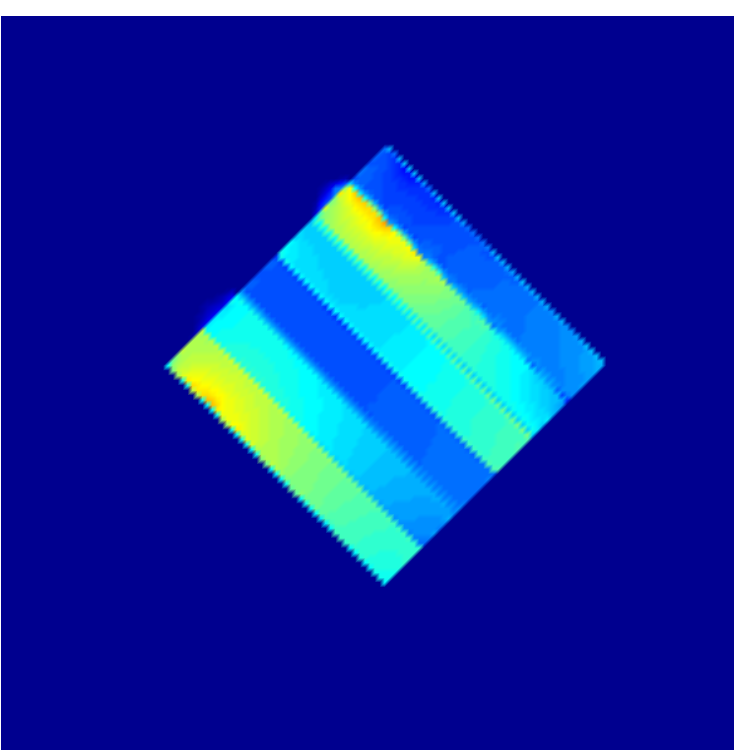}&
\includegraphics[height=2cm]{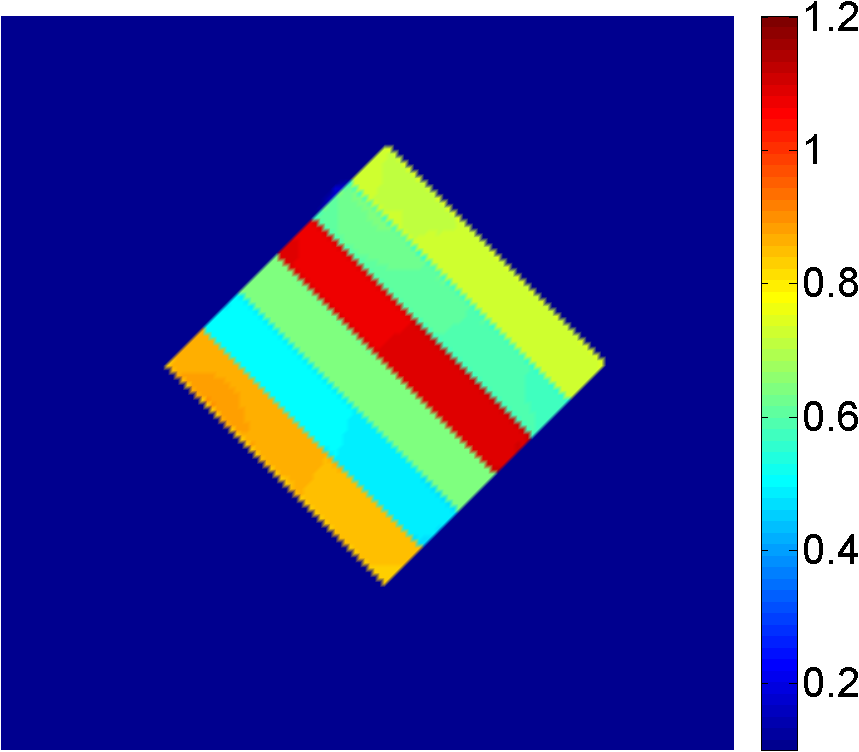}\\
\includegraphics[height=2cm]{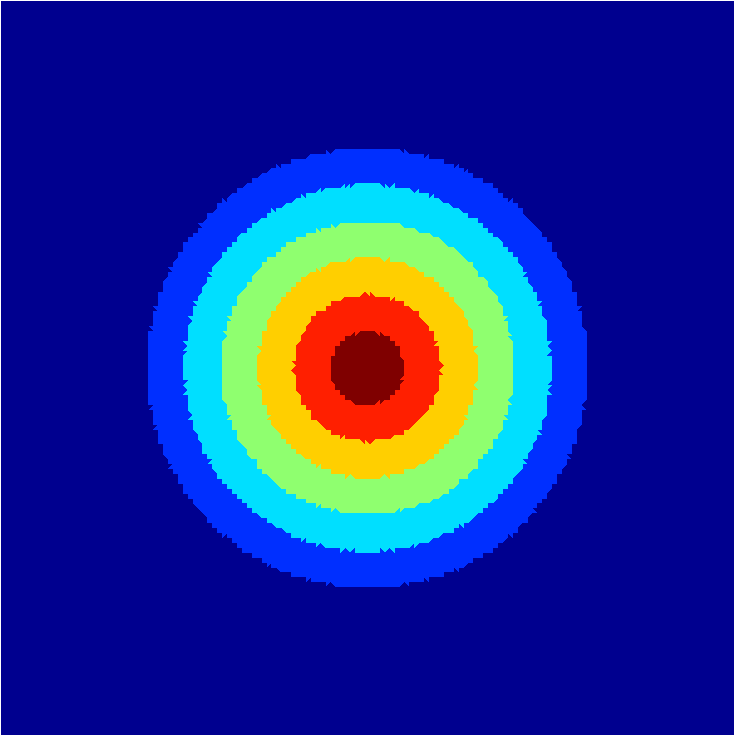}&
\includegraphics[height=2cm]{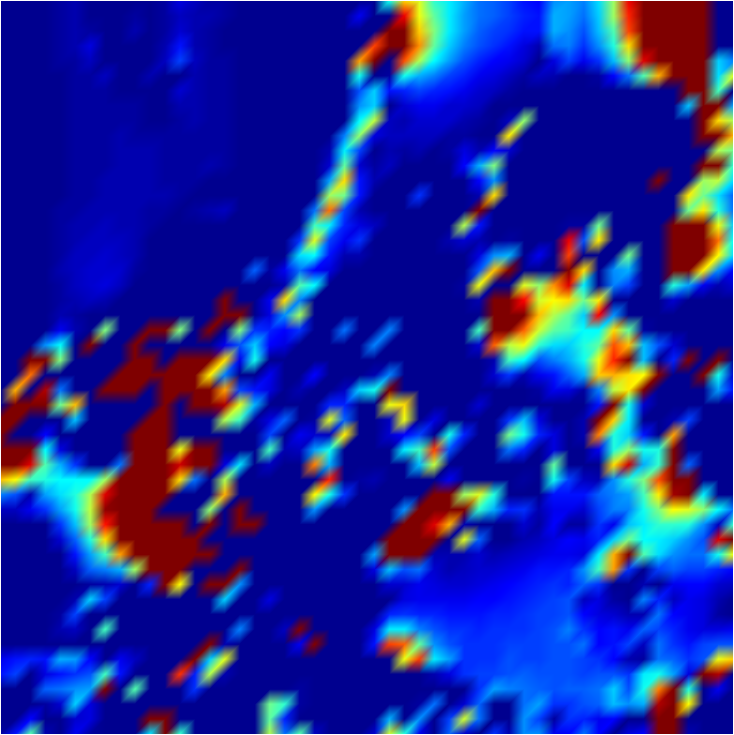}&
\includegraphics[height=2cm]{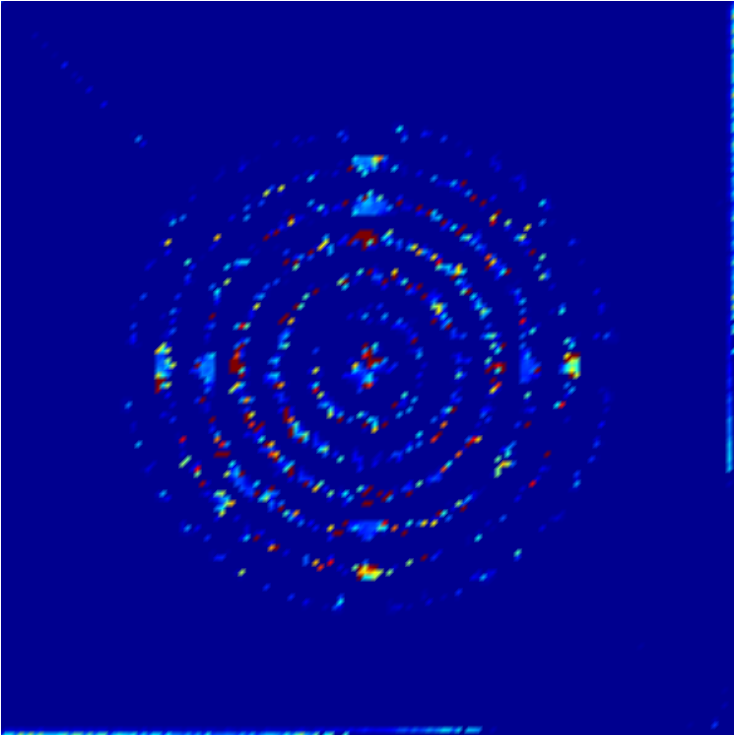}&
\includegraphics[height=2cm]{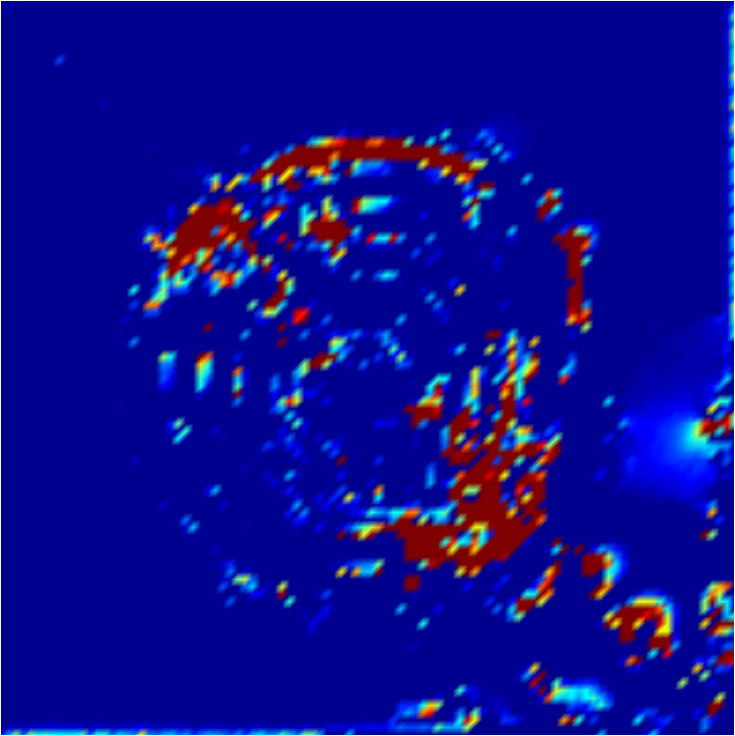}&
\includegraphics[height=2cm]{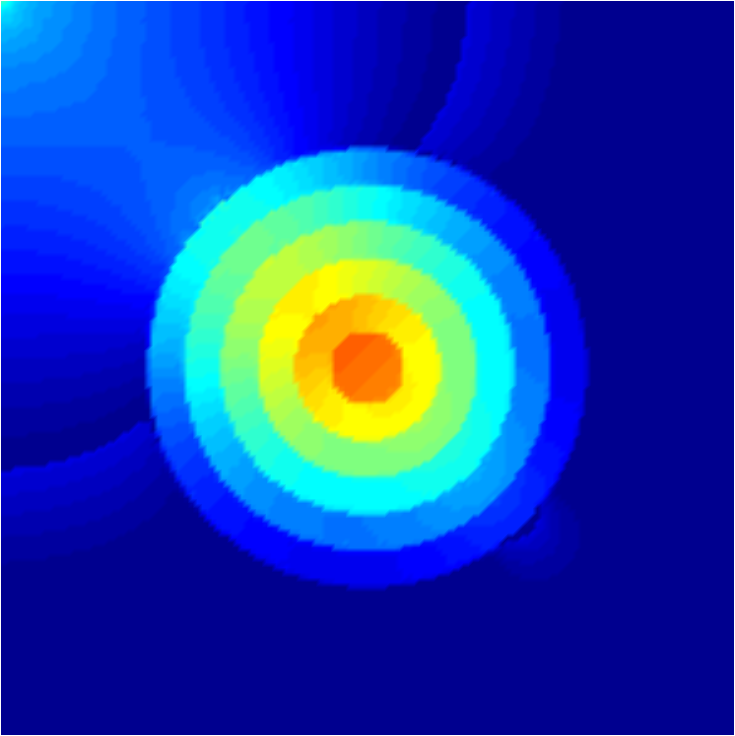}&
\includegraphics[height=2cm]{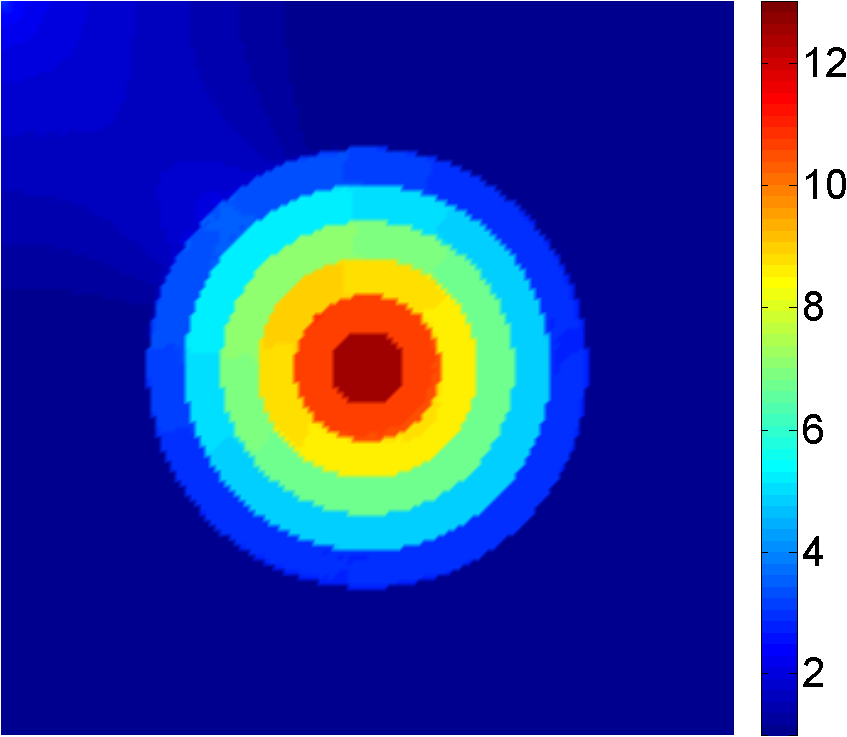}\\
\includegraphics[height=2cm]{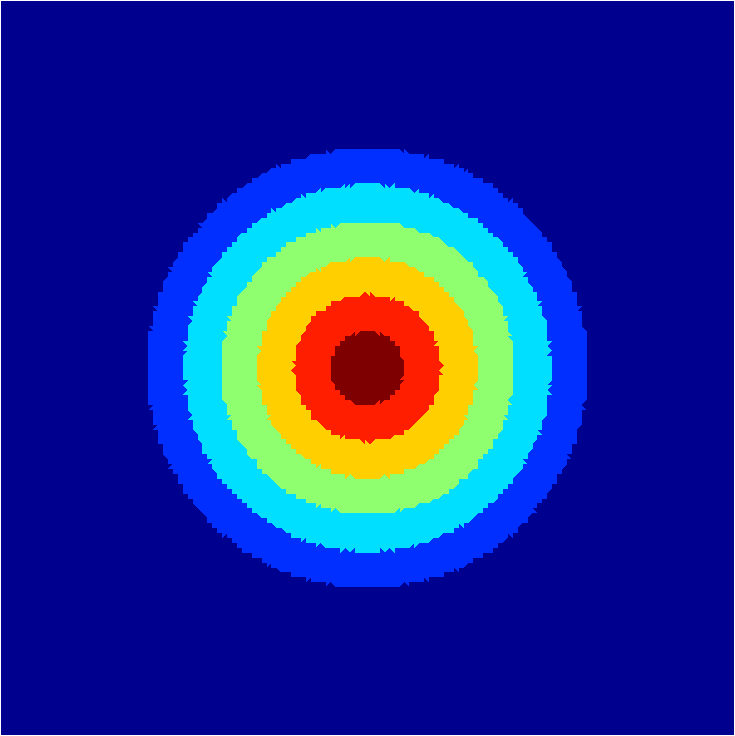}&
\includegraphics[height=2cm]{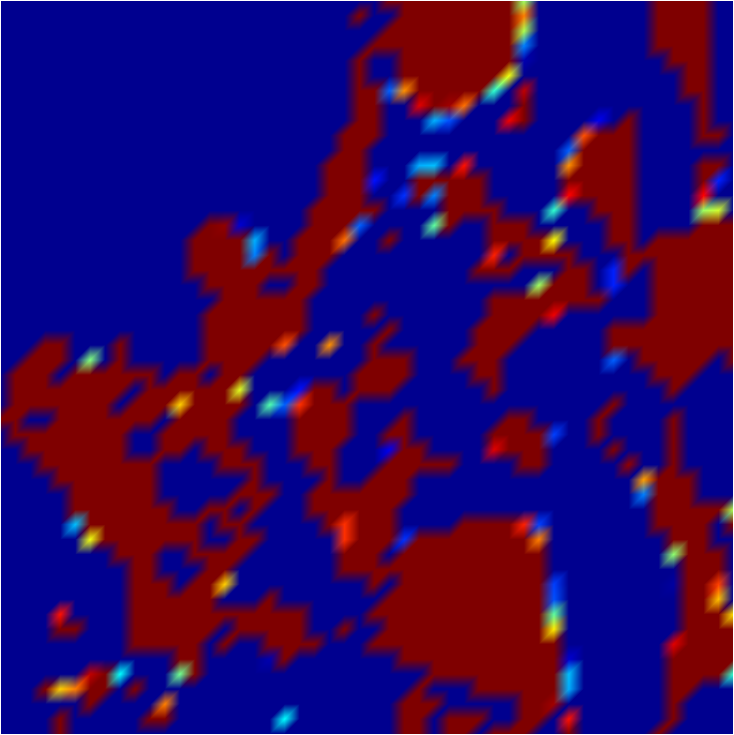}&
\includegraphics[height=2cm]{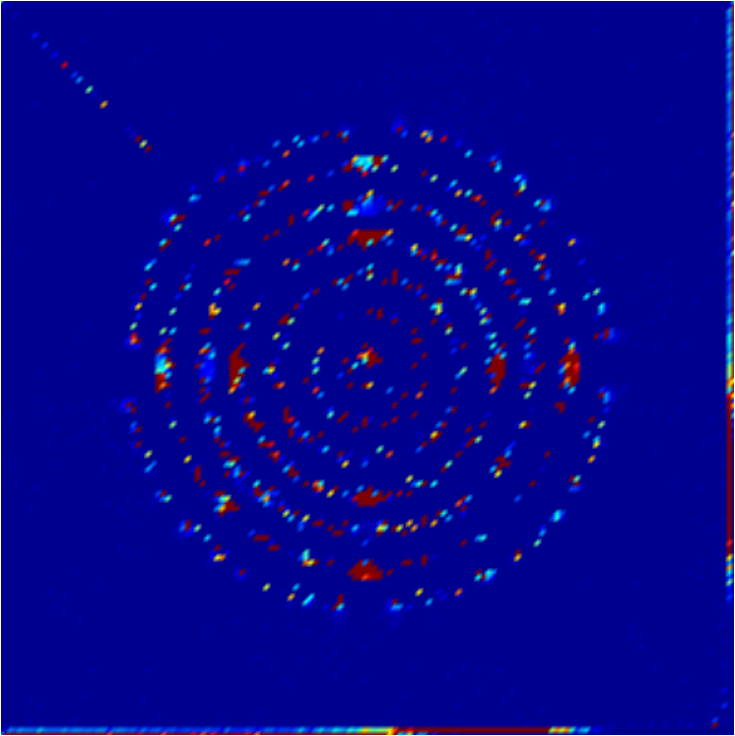}&
\includegraphics[height=2cm]{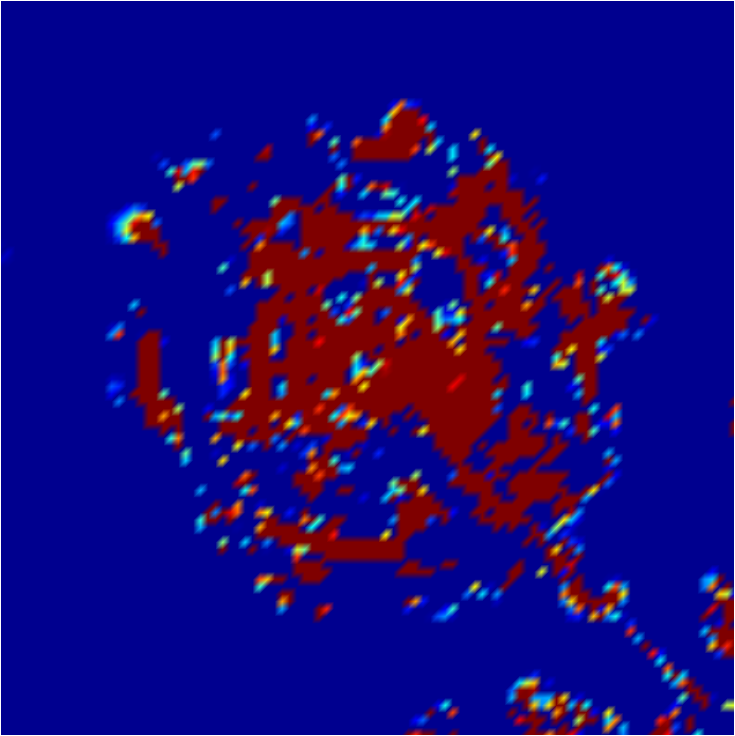}&
\includegraphics[height=2cm]{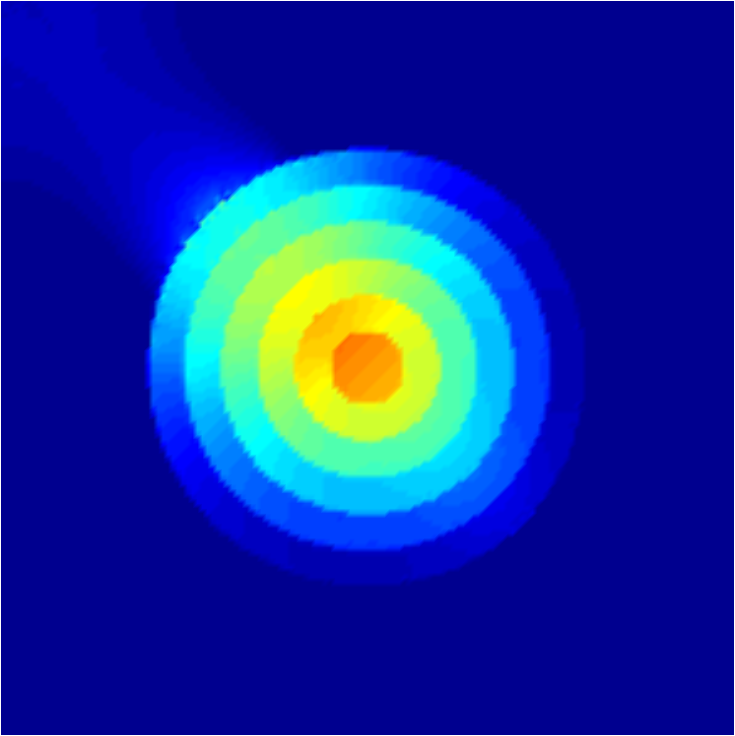}&
\includegraphics[height=2cm]{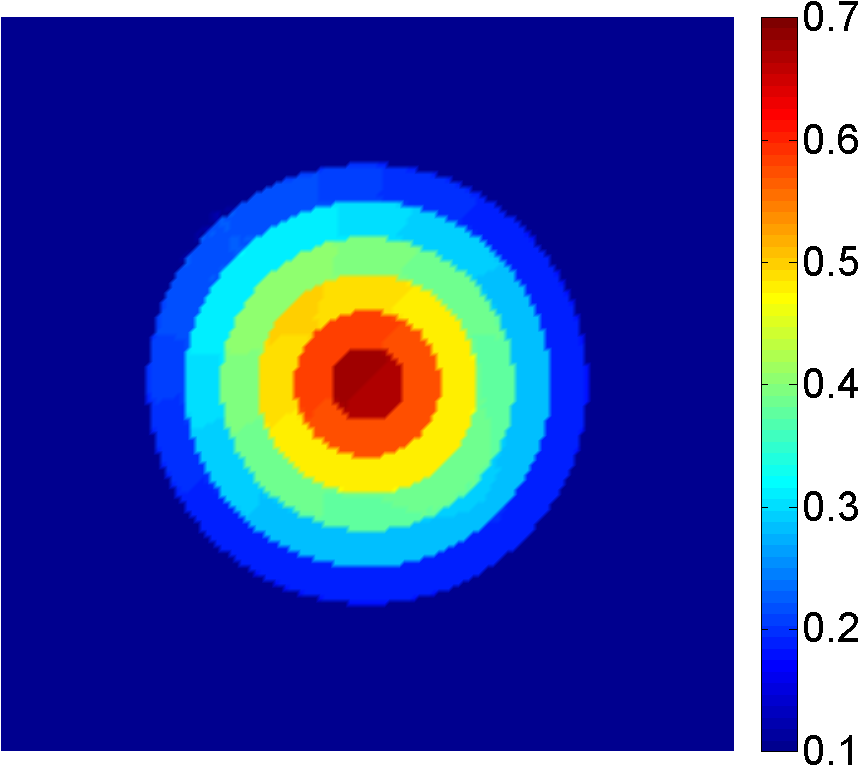}\\
(a) & (b)&(c)&(d)&(e)&(f)
\end{tabular}
\caption{Simulation results for viscoelasticity image reconstruction. First, third and fifth rows: images of $\mu$. Second, fourth and sixth rows: images of $\eta_\mu$. (a) True images; (c) direct inversion method; (e) hybrid one-step method; (b), (d) and (f) are reconstructed images by the adjoint-based optimization method \eref{Eq:iterationscheme} with initial guess of the constant $\mu_0+i\omega\eta_{\mu_0}$, (c) and  (e), respectively.}
\label{Fig-reconstruction-model123}
\end{figure}

We also numerically evaluate the local reconstruction method proposed in section \ref{Section-local}. We consider the rectangular domain, $\Omega$, which is equally divided into four parts: top-left, top-right, bottom-left and bottom-right. It is assumed that the top-right part is contaminated by noise or defected data. For numerical simplicity, we add  3\% white noise to the measured data in the top-right part. The reconstruction results in both the whole domain  and the local domains  are shown in  figure \ref{Fig-local} where (a) is the true distribution of shear viscoelasticity, (b) the initial guess with hybrid method, (c) the reconstruction in whole domain using proposed method,
(d) the local reconstruction.
\begin{figure}[!h]
\centering
\setlength{\tabcolsep}{2pt}
\begin{tabular}{ccccc}
\includegraphics[height=2.5cm]{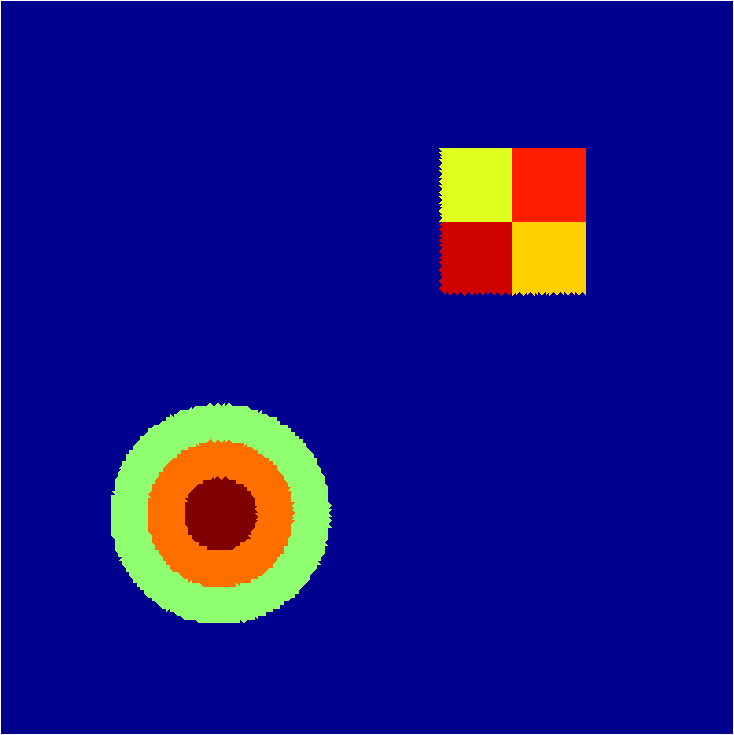}&
\includegraphics[height=2.5cm]{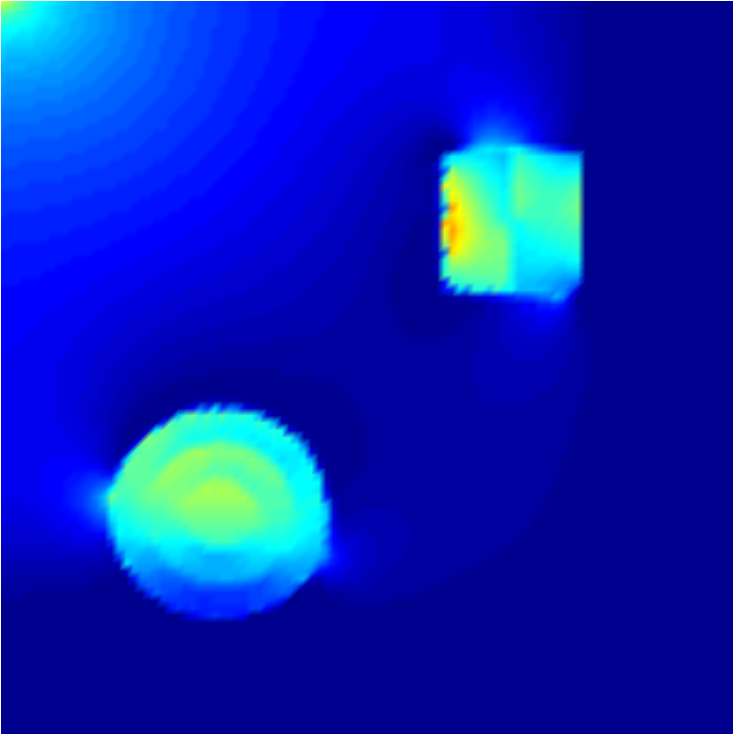}&
\includegraphics[height=2.5cm]{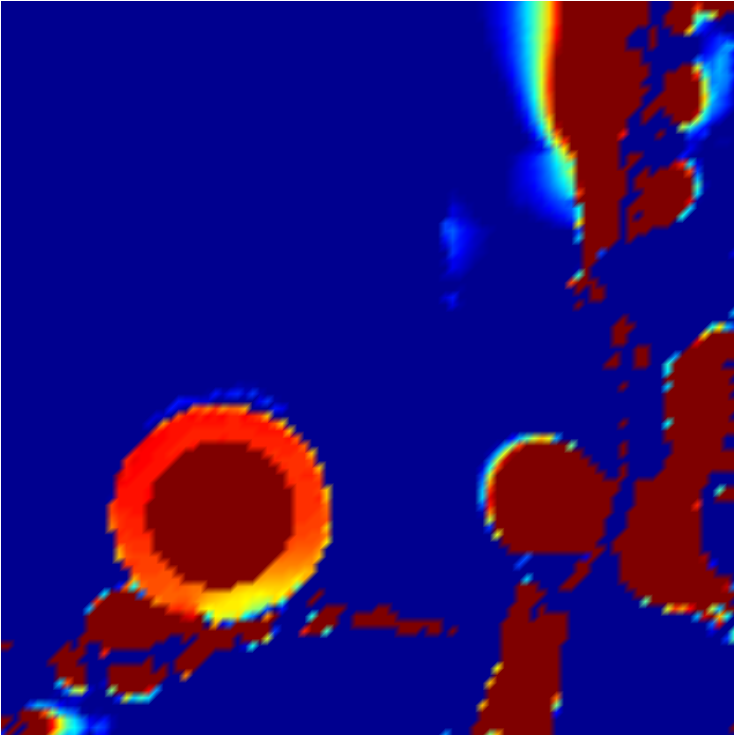}&
\includegraphics[height=2.5cm]{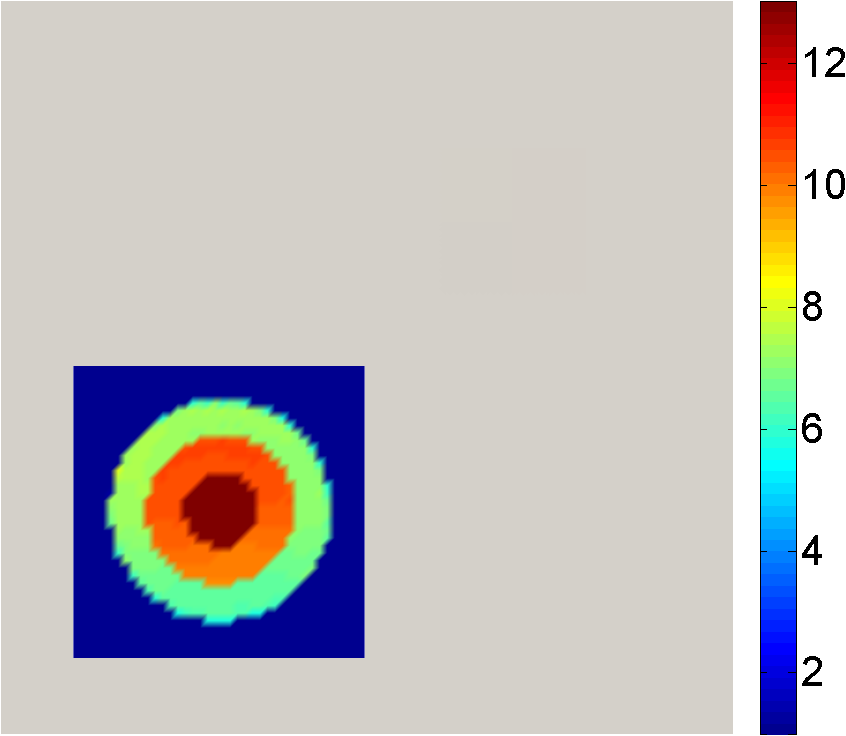}\\
\includegraphics[height=2.5cm]{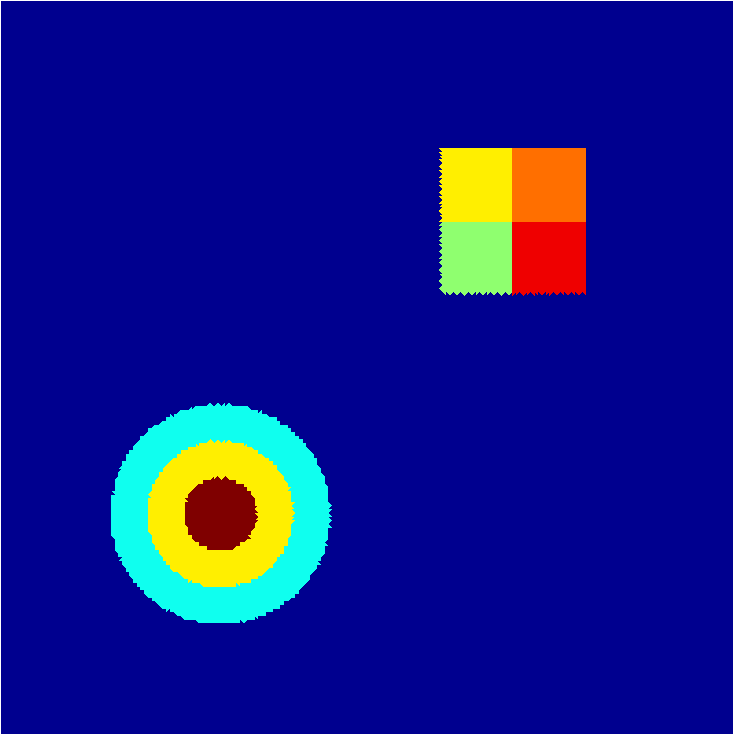}&
\includegraphics[height=2.5cm]{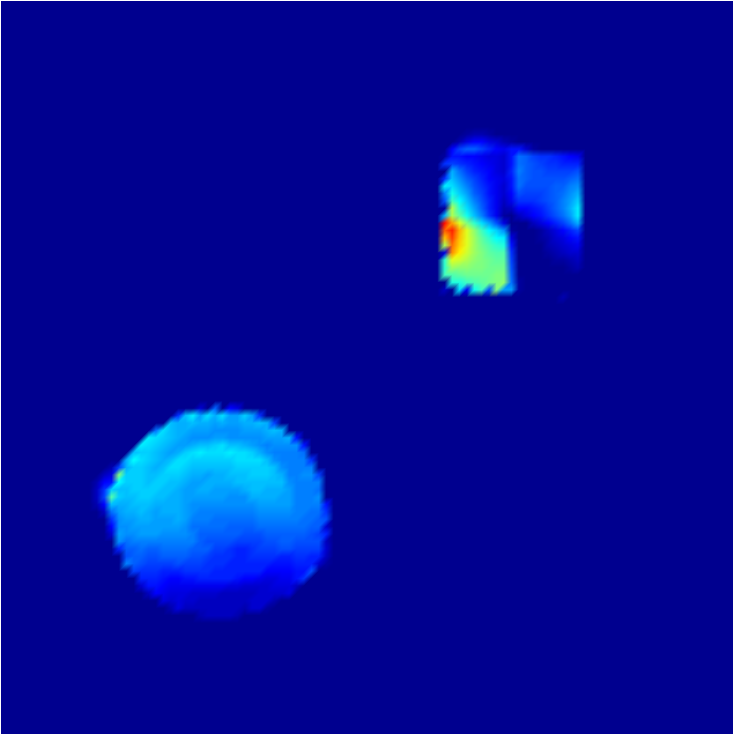}&
\includegraphics[height=2.5cm]{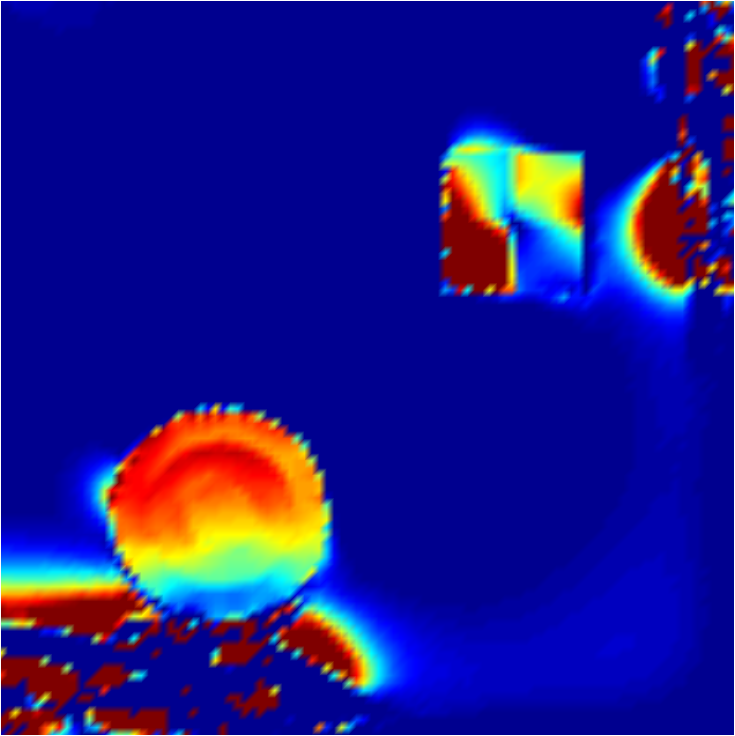}&
\includegraphics[height=2.5cm]{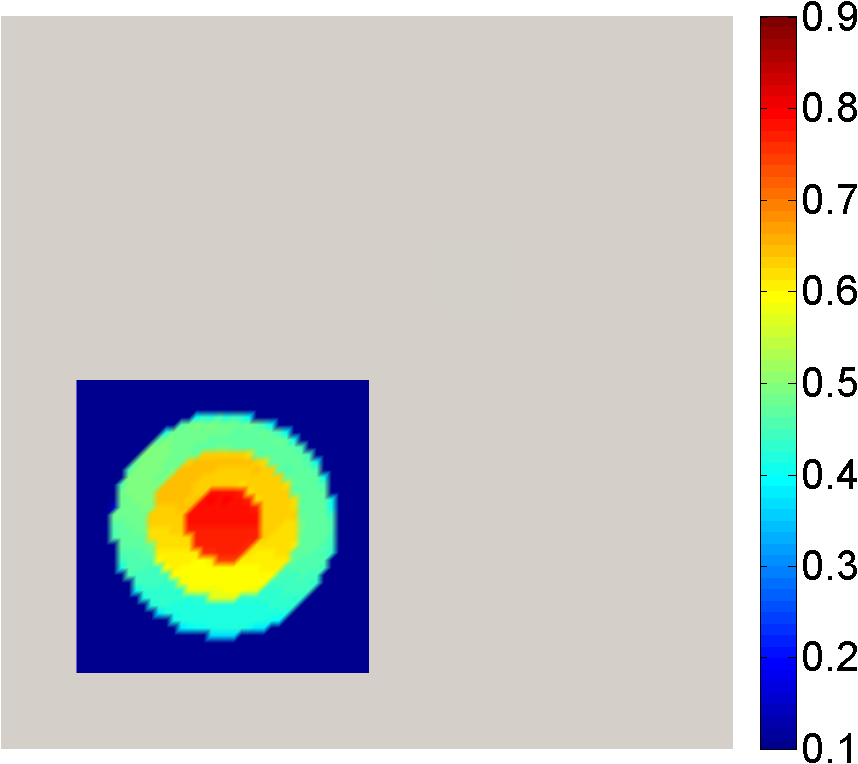}\\
(a)&(b)&(c)&(d)
\end{tabular}
\caption{Simulation results for local reconstruction. First row: images of $\mu$. Second row: images of $\eta_\mu$. (a)  true image; (b) initial guess; (c) adjoint-based optimization method; (d) local reconstruction. }
\label{Fig-local}
\end{figure}

\section{Conclusion}
In this paper, we propose a reconstruction algorithm for shear elasticity and shear viscosity in a viscoelastic tissue. Our optimization-based approach involves introducing an adjoint problem to avoid taking any derivative of the measured time-harmonic internal data. The proposed initial guess formula is particularly  suitable for imaging viscoelastic inclusions. The local convergence of the developed optimal control approach is an open problem. The recent stability results in \cite{otmar} may be helpful in solving this difficult question. It would be also very interesting to generalize the proposed method for imaging anisotropic viscoelastic media. Another challenging problem is to recognize the disease state in tissue from multifrequency elastographic measurements. These important problems will be the subject of future work.

\section*{Acknowledgements}

Ammari was supported  by the ERC Advanced Grant Project MULTIMOD--267184. Seo and Zhou were supported by the National Research Foundation of Korea (NRF) grant funded by the Korean government (MEST) (No. 2011-0028868, 2012R1A2A1A03670512).

\end{document}